\documentclass[preprint,11pt]{elsarticle}

\journal{JMAA}

\usepackage{amsmath}
\usepackage{amssymb}
\usepackage{amsthm}

\newtheorem{theorem}{Theorem}
\newtheorem{proposition}{Proposition}
\newtheorem{lemma}{Lemma}
\newtheorem{corollary}{Corollary}

\newtheorem{definition}{Definition}
\newtheorem{example}{Example}
\newtheorem{remark}{Remark}

\begin{document}

\begin{frontmatter}

\title{Zero Attractors of Partition Polynomials}

\author[1]{Robert P. Boyer
}
\ead{boyerrp@drexel.edu}

\author[2]{Daniel T. Parry}
\ead{daniel.parry@jpmchase.com}

  \address[1]{Department of Mathematics, Drexel University
  \\
  Philadelphia, PA 19104}
  
  \address[2]{Model Risk Governance and Review, JP Morgan Chase,
  \\          
  4 Chase, MetroTech Center 11245
New York, New York 
}

\begin{abstract}
A partition polynomial is a refinement of the partition number $p(n)$ whose coefficients count some special
partition statistic. Just as partition numbers have useful asymptotics so do partition polynomials.
In fact, their asymptotics determine the limiting behavior of their zeros which form a network of curves
inside the unit disk. An important new feature in their study  requires a detailed analysis of. the ``root dilogarithm"
given as the real part of the square root of the usual dilogarithm.
\end{abstract}


\end{frontmatter}

\section{Introduction }
More than twenty years ago,  Richard Stanley gave a talk on the geometry of the roots of various polynomial sequences
where the roots   appear to lie on  complicated unknown curves in the complex plane.
The examples he emphasized   are  influential within combinatorics and include the Bernoulli polynomials \cite{ MR881502}, chromatic polynomials \cite{MR2047238}, and $q$-Catalan numbers.  
The list contains a single example  from the theory of integer partitions  \cite[Chapter 2]{MR1634067}
\[
F_n(z) = \sum_{k=1}^{n}p_k(n) z^k,
\]
which we now call the partition polynomials. Here the number $p_k(n)$ counts the total number of integer partitions of $n$ with $k$ parts.  This polynomial sequence appeared already in several areas of mathematics and physics.  
First, E.M. Wright studied $F_n(z)$ to develop the ``Wright circle method" to produce estimates for $F_n(z)$ when $z>0$ \cite{MR1575956} under the notion of ``weighted partitions."  This paper  has been an inspiration to many works in combinatorics and number theory (see \cite{MR3423418} for example). Second, Navez et al. 
discovered that $F_n(z)$ can define a probability measure called the ``Maxwell's Demon Ensemble" which approximates a Bose gas near certain temperatures \cite{Weiss:97, Grossmann:97, PhysRevLett.79.1789}.   

With Richard Stanley's talk in mind, Boyer and Goh \cite{MR2427666,MR2367410} initiated  a study of the roots of $F_n(z).$  They announced the results of  their study but they never published proofs. 
 In fact they confirmed Richard Stanley's suspicion that there are indeed  curves that attract the roots of the partition polynomials. 

For background, the Wright circle method produces an asymptotic estimate of $F_n(z)$ as $n\to \infty,$
 which requires the choice of a ``heaviest" or ``dominant" singularity of a bivariate generating function $P(z,q).$ 
 It turns out that for $z>0,$  this heaviest singularity  always occurs at $q=1$ but for general $z$ in the unit disk $\mathbb{D}$
  the heaviest singularity depends on the location of $z.$  
  Since the asymptotics for $F_n(z)$ depends primarily on the heaviest singularity, a Stokes phenomenon \cite{MR1012299} appears  forcing the  roots fall  near the Stokes/anti-Stokes lines.  This is analogous to the phenomenon observed by the well studied Szego curve \cite{MR1407500} and the Appell polynomials \cite{MR2288197,MR2776473}. 

Boyer and Goh also observed that the dilogarithm $Li_{2}(z)= \sum_{n=1}^\infty z^n /n^{2}$ plays an important role in determination of  the heaviest singularity.  In itself,  this is not unusual in integer partition theory because the dilogarithm is intricately connected to partitions (\cite{MR2471621,MR940391,MR3334080} for example) but what is unusual is that overcoming
the additional challenge of allowing $z$ be a complex number
required univalent function theory including conformal properties of the dilogarithm \cite{MR3283658,MR2743537,MR697137}.

With that, Boyer and Parry began the study of a similar polynomial sequence called the plane partition polynomials \cite{MR3370706, MR3216005, MR2880680} whose study was suggested by Stanley in conversation in 2005.
  This polynomial sequence acts in the same way as Boyer and Goh predicted 
   with the different ``phases" of the asymptotic structure of $F_n(z)$ driving the location of the roots of $F_n(z).$  It is now clear that the 
  partition polynomials are not an isolated case but part of an entire family of polynomial sequences  in combinatorics that have this Stokes phenomenon.  

Construction of a general framework for partition polynomials began with \cite{MR3280954} when Parry
 found a polynomial version of the classical Meinardus Theorem for the asymptotics of sequences whose generating functions have a special (Euler)  infinite product structure.  This structure is common in the classical study of integer partitions and given one knows the phase structure, one can determine the asymptotics of the polynomial sequence.   

This paper finishes the framework of Boyer and Parry and confirms the conjectures of Boyer and Goh.
Specifically, we will show that the same phenomenon observed by the original partition polynomials and established for the plane partition polynomials occurs  for  the partition polynomials for arithmetic progressions.

\section{Mathematical Preliminaries}

Partition polynomials have been studied in a few contexts \cite{MR3351540, MR1502817} and loosely define a family of polynomial generating functions that have an interpretation in integer partition theory.  For our purposes, we will use the following definition:

\begin{definition} 
The  partition polynomials, $F_n(z)$, are defined as the Fourier coefficients of  
\begin{equation}\label{eq:gen_fct}
1+ \sum_{n=1}^\infty F_n(z)q^n 
=
\prod_{m=1}^\infty \frac{1}{(1-zq^m)^{a_m}} 
\end{equation}
where the integer exponents $a_m$ are  either $0$ or $1$.
\end{definition}

For $S \subset {\mathbb Z}^+ $, let $a_m = \chi_S(m)$, 
then the coefficient $[z^k] F_n(z)$ counts the number of partitions of $n$ with $k$
parts that all lie in $S$.  The principal cases we study are
  $S = {\mathbb Z}^+$; that is, the parts have no restriction;
 and where the elements of $S$ satisfy a linear congruent equation modulo $p$.

We recall the notion of phase \cite[Definition 1]{MR1634067} for a
 sequence $\{ L_n(z)\}$ of functions on the open subset $U$ of the complex plane where $L_n(z)$
 is analytic on $U$ except perhaps for a branch cut.
 
 \begin{definition}
 Assume 
$U$ has a decomposition as a finite  disjoint union of open nonempty sets 
$R(m_1),  \allowbreak \cdots,  \allowbreak R(m_k)$
together with their boundaries
 such that
 \begin{equation}\label{eq:phase_definition}
 \Re L_{m_j}(z) > \max\{ L_{m}(z) : m \neq m_j \}
 \end{equation}
 for $z \in R(m_j)$.
 If  the family of sets $R(m)$
is a maximal family with these properties, we call $R(m_j)$
a  phase with phase function $L_{m_j}(z)$.
 \end{definition}

The term phase comes from statistical mechanics;  in that context a phase function
represents a metastable free energy. In our examples, $U$ will either be the complex plane $\mathbb C$
or the open unit disk $\mathbb D$.  By maximality, the phase functions are not analytic continuations of each other.

Our description of the asymptotics for  partition polynomials $\{ F_n(z)\}$ inside the open unit
disk $\mathbb D$
uses the notion of phase and phase functions.
On $\mathbb D$, we find that  there is   a sequence
$\{ L_m(z)\}$ of functions 
and  finitely many indices $m_1, m_2, \dots, m_\ell$, say,
so on each  $R(m_j)$
\[
F_n(z) \sim_X A_{n,m}(z) \exp( n^{- \beta} L_{m_j}(z) ), \quad \beta>0,
\]
where $A_{n,m}(z)$ are nonzero functions on $R(m_j)$
such that 
\[
\lim_{n\to\infty} n^{-\beta} \ln \vert  A_{n,m}(z) \vert = 0
\]
 uniformly on the compact subsets of $R(m_j)$,
$L_{m_j}(z)$ is analytic on $R(m_j)$ except perhaps for a branch cut, 
and $\beta=1/2$.

Next we recall the definition of principal object of this paper.

\begin{definition} 
Let $Z(P(z))$ denote the finite set of zeros of the polynomial $P(z)$.
For a polynomial sequence $\{ F_n(z)\}$ whose degrees go to infinity,
its  zero attractor  $A$  is the limit of $Z(F_n(z))$ in the 
 Hausdorff metric  on the non-empty compact subsets $K$ of $\mathbb{C}\cup \{\infty\}$
 \cite{MR2288197}.
\end{definition}

We now discuss fully  below the connection between phases and the zero attractor
For a partition polynomial sequence $\{ F_n(z)\}$
on the unit disk, if its
phases there are,  $R(m_1), \allowbreak \dots, \allowbreak R(m_\ell)$ say, then 
its zero attractor $A$ is simply the union of the boundaries of the phases 
\[
A = \partial R(m_1) \cup \cdots \cup \partial R(m_\ell) 
\]
with perhaps further  contributions from branch cuts if necessary.

Finally we mention that we  continue using our  notational conventions 
  in \cite{MR1634067}. 
  So $\sqrt[s]{x}$ is  the $s$ root $\exp(\frac{\log{x}}{s})$
   with the imaginary part of the logarithm defined on $(-\pi,\pi].$ 
   Both  $[x]^-$ and $\overline{x}$ denote the complex conjugate of $x.$
    Let $\{g_n(x)\}$ a sequence of functions,
     we say $g_n(x)=O_V(a_n)$ where $a_n$ a sequence of complex numbers, 
    if there exists a constant $C_V$ dependent solely on a collection of parameters $V,$ 
    such that $\vert g_n(x)\vert \leq C_V \vert a_n \vert$ as $n\to \infty.$   
    Similarly,
    we say $g_n(x)=o_V(a_n)$ for every $C_V$ dependent solely on a collection of parameters $V,$ 
    $\vert g_n(x)\vert\leq C_V \vert a_n\vert$ as $n\to \infty.$
    Absence of any $V$ indicates that the constant is uniform.  
    
    For convenience, we use the notation $e_k(z)=\exp(2\pi i z/k)$ and $e(z)=e_1(z).$

\subsection{Summary of the Meinardus Polynomial Setup  \cite{MR3280954}} 

Let $\{ a_m\}$ be the  exponent sequence  for the generating function for the partition polynomials $\{F_n(z)\}$. For each positive
integer $k$
and character of ${\mathbb Z}_k$, 
we associate the Dirichlet series $D_{t,k}(s)$ where
$t$ is an integer $1\leq t \leq k$ that indexes the character of the finite
cyclic group ${\mathbb Z}_k$ given by $j \mapsto e_k( tj)$.
We define the Dirichlet series as
\[
D_{t,k}(s) = \sum_{m=1}^\infty a_m \frac{ e_k( mt) }{ m^s } .
\]
Observe that $D_{k,k}(s)$ is the usual Dirchlet generating function for the sequence $\{ a_m \}$.
In the examples  in this paper, these Dirichlet series either   have
an analytic continuation to the complex plane as an entire function or as
a meromorphic function with a unique singularity at $s=1$, which will be a simple pole.
Further, for each positive integer $k$, we need to  introduce  two functions on the character group 
$\widehat{ {\mathbb Z}_k}$ as well as their finite Fourier transforms 
given in terms of the Dirichlet series family.
\begin{definition}\label{def:b_k,a_k}
Let $b_k$ be the $\widehat{{\mathbb Z}_k}$-Fourier transform of 
$
t \mapsto D_{t,k}(0)
$
while
$c_k$ be the $\widehat{{\mathbb Z}_k}$-Fourier transform of 
 $
 t \mapsto Res( D_{t,k}, s=1)  
 $; explicity, for $1 \leq j \leq k$,
 \[
 b_k( j ) = \frac{1}{k} \sum_{t=1}^k e_k( - t j) D_{t,k}(0),
 \quad
 c_k( j ) = \frac{1}{k} \sum_{t=1}^k e_k( - t j) \, Res( D_{t,k}(s), s=1) .
 \]
\end{definition}
The actual form of the asymptotics of $\{ F_n(z)\}$ is given in terms of these functions
on ${\mathbb Z}_k$ and its dual.
For $k,n \in {\mathbb Z}^+$ and $1\leq h <k$ relatively prime to $k$, define 
$L_{h,k}(z)$ and $\omega_{h,k,n}(z)$:
\begin{align*}
L_{h,k}(z)^2
&=
\frac{1}{k^2} \sum_{r=1}^k z^r L( z^k,2,r/k) \,  Res( D_{rh,k},s=1)
=
\sum_{j=1}^k c_k(j) Li_2( e_k( j h) z) 
 ,
\\
\omega_{h,k,n}(z)
&=
e_k( -h n )
\prod_{j=1}^k (1- e_k(h  j)   z)^{- b_k(j)} , 
\quad \Omega_{n,k }(z)= \sum_{  (h,k)=1} \omega_{h,k,n}(z) 
\end{align*}
where $L(z,s,\nu)$ is the Lerch zeta function.
When the functions $L_{h,k}(z)$ are independent of $h$, we will
suppress the $h$ dependence in the indexing when convenient.

The asymptotic results in \cite{MR3280954} depend on knowing the phases for the function sequence
$\{ L_{h,k}(z)\}$.
For simplicity of exposition, we  assume that  $L_{h,k}(z)$ are independent of $h$ and that the
phases are $R(m_1), \dots, R(m_j)$. On each phase
 $R(m)$ there are possibly two different regimes of the asymptotics on
compact subsets $X$:
\begin{align}
&F_n(z)  \label{eq:form1}
\sim_X
\frac{1}{2 \sqrt{\pi} n^{3/4}}  \Omega_{n,m}(z)    \sqrt{ L_m(z)} \exp( 2 n^{1/2} L_m(z) ),
\,
z \in X \setminus \{ z : L_m(z) \leq 0\},
\\
\label{eq:form2}
&\sim_X
2 \Re\left[  
\frac{1}{2 \sqrt{\pi} n^{3/4}}  \Omega_{n,m}(z)   \sqrt{ L_m(z)} \exp( 2 n^{1/2} L_m(z) )
\right],
\,
z \in X \cap \{ z : L_m(z) \leq 0\},
\end{align}
where the second asymptotics represent the contribution of a branch cut, if needed.

\subsection{General Results about Zero Attractor} 

We show that the zero attractor lies inside the closed unit disk $\overline{\mathbb D}$ and always contains the unit circle
for any family of partition polynomials
if their exponent sequence $\{a_n\}$ satisfies the initial condition $a_1=1$
and  contains infinitely many $1$'s.
 Of course, the more detailed structure of  the zero
attractor depends on the moduli of $\{ F_n(z)\}$ so the chief contribution to consider is $\exp( 2 n^{1/2} \Re L_m(z))$
in its asymptotics.

We restate Sokal's result in language convenient for partition polynomials.

\begin{proposition}
\label{prop:sokal}
\cite{MR2047238}
Let $\{  \phi_n(z)  \}$ be a sequence of analytic functions on 
an open connected set $U$. 
For a fixed $\beta >0$,
we assume that the sequence $ \{  \vert  \phi(z) \vert^{ {n^\beta}  }  \} $ is uniformly bounded
on the compact subsets of $U$. A point $z_0\in {\mathbb C} $ is in their zero
attractor if there exists a neighborhood $V(z_0)$ of $z_0$
with the property that there cannot exist a harmonic function $v(z)$
on $V(z_0)$ satisfying
\[
\liminf_{n\to\infty} n^{-\beta} \ln \vert \phi_n(z) \vert
\leq
v(z)
\leq
\limsup_{n\to\infty} n^{-\beta} \ln \vert \phi_n(z) \vert.
\]
\end{proposition}

For partition polynomials, there are two choices for the scales $\beta = 1/2$ or $1$.
In  next Proposition, we make no assumptions on the phase structions inside the unit disk.

\begin{proposition}
Let $\{ F_n(z)\}$ be a sequence of partition polynomials. Then 
\\
(a)
$\vert F_n(z) \vert^{1/\sqrt{n}}$ is uniformly bounded
on the  unit disk.
\\
(b) Uniformly on compact subsets of $\mathbb C \setminus \overline{ \mathbb D}$:
\[
\lim_{n\to\infty} \frac{1}{n} \ln \vert F_n(z)\vert 
=
\ln\vert z \vert .
\]
(c) \label{corollary:outerasym} 
Let   $\{a_m\}_{m=1}^\infty$ be an exponent sequence for the 
partition polynomial sequence $\{ F_n(z)\}$
with $a_1=1$.
 Then its zero attractor $A$ 
lies inside the closed unit disk  and $\infty \notin A$.
\end{proposition}
\begin{proof}
(a)
Since the exponents $\{ a_m\}$ form a binary $0$-$1$ sequence,   $F_n(1)$ is bounded above by  the total number $p(n)$  of the standard partitions of $n$. But
$p(n)$ has order  $\exp( \pi \sqrt{ 2n/3})$. 
\\
(b)
 Assume $\vert z \vert>1$. Then
\begin{align*}
\frac{1}{n} \ln \vert F_n(z)\vert
&=
\frac{1}{n}
n \ln \vert z \vert
+
\frac{1}{n}
\ln \left\vert 1+ ([z^{n-1}]F_n(z)) /z + \cdots + ([z^{1}]F_n(z)) / z^{n-1} \right\vert
\end{align*}
But
\begin{align*}
&
\frac{1}{n}
\ln \left\vert 1+  ([z^{n-1}]F_n(z)) /z + \cdots + ([z^{1}]F_n(z)) / z^{n-1} \right\vert
\leq \frac{1}{n} \ln F_n(1)
\\
& \qquad
\leq \frac{1}{n} \ln p(n)  \to 0
\end{align*}
as $n\to\infty$. Hence $\ln \vert F_n(z)\vert/ n \to \ln\vert z \vert$.

\noindent
(c)
Let $H(z) = \prod_{m=1}^\infty { (1-z^m)^{-a_{m+1}}}$ so $H(z)$ is analytic,  nonvanishing in the open unit disk
$\mathbb D$, and $H(0)=1$.  By the Stabilization result in \cite{MR3106102}, we find that
\[
[z^{n-k}] F_n(z) = [z^k] H(z), \quad 0 \leq k < \lfloor n/2 \rfloor .
\]
Write $H(z) = 1+ \sum_{m=1}^\infty h_m z^m$. Let $1<r_0$. For $\vert z \vert \geq r_0$, we find there is
a cancellation of the first $ \lfloor n/2 \rfloor$ terms below to give the bound:
\begin{align*}
\left\vert
\frac{ F_n(z)}{ z^n} - \prod_{m=1}^\infty ( 1- 1/z)^{ - a_{m+1}}
\right\vert
&\leq
\sum_{\ell= \lfloor n/2 \rfloor}^n  [z^{n-\ell}] F_n(z) \frac{1}{ \vert z\vert^\ell} 
+
\sum_{\ell =  \lfloor n/2 \rfloor+1}^\infty h_\ell \vert z \vert^{ - \ell} 
\\
&\leq
r_0^{ \lfloor n/2 \rfloor } F_n(1)
+
\sum_{\ell =  \lfloor n/2 \rfloor+1}^\infty h_\ell  r_0^{ - \ell} .
\end{align*}
Hence, on $\vert z \vert \geq r_0 >1$, $F_n(z)/z^n$ converges uniformly  to $H(1/z)$.
Let $\{ z_{n_k}\}$ be a convergent sequence of zeros of $F_{n_k}( z_{n_k})=0$ such that $\vert z_{n_k}\vert \geq r_0$ with limit $z^*$.
 Then $F_{n_k}( z_{n_k}) /z_{n_k}^{n_k} = 0 \to  H( 1/z^*)  \neq 0$ which ontradicts
that  $H(1/z)$ does not vanish on $\mathbb D$.
On the other hand,
 $\vert z _{n_k}\vert \to \infty$ forces $H(0)=0$  but this contradicts $H(1)=0$.

\end{proof}

\begin{theorem}\label{theorem:asymptotic_form}
Let $\{ F_n(z)\}$ be a partition polynomial sequence
whose phases are the disjoint nonempty open subsets  are 
$R(m_1), \cdots, R(m_\ell)$ inside $\mathbb D$.
Assume the functions $\Omega_{n,m_j}(z) \neq 0$ for $0 \neq z \in R(m_j)$.
\\
(a)
 Then any $z_0 \in \partial R(m_j) \cap {\mathbb D}$
is in the zero attractor of $\{ F_n(z) \}$.
\\
(b) No nonzero point in $R(m)$ lies in the zero attractor.
\\
(c) The unit circle lies in the zero attractor.
\end{theorem}
\begin{proof}
(a)
To apply Sokal's result  (see Proposition \ref{prop:sokal}), we need to consider the limit
\[
\lim_{n\to\infty} \frac{1}{ \sqrt{n}} \ln \vert F_n(z) \vert = 2 \Re L_m(z), \quad z \in R(m),
\]
that holds uniformly on compact subsets of the phase $R(m)$. By assumption, $ \Re L_m(z) $ is harmonic on $R(m)$ and
for $\ell \neq  k$, $  \Re L_\ell(z) $ and $\Re L_k(z)$ are not harmonic continuations of each other. 
Hence, for $z_0 \in {\mathbb D} \cap \partial R(m)$, there cannot be a harmonic function $v(z)$ in a neighborhood
of $z_0$ such that $\liminf \ln \vert F_n(z) \vert \leq v(z) \leq \limsup \ln \vert F_n(z)  \vert$. We conclude that $z_0$ lies in the
zero attractor.

\noindent
(b) By assumption, we know that
\[
\lim_{n\to \infty} n^{-3/4} F_n(z)  \exp( -2 n^{1/2} L_m(z) ) = \frac{1}{2 \sqrt{\pi} }  \Omega_{n,m}(z)    \sqrt{ L_m(z)}
\]
uniformly on the compact subsets of $R(m)$. Further, this limit is nonzero by construction. Hence $z$ cannot lie
in the zero attractor.

\noindent
(c)
Let $z^*$ be a point on the unit circle. Consider the open disk  $V$ with center $z^*$ and radius 
$r>0$.  
We consider the limit
\[
\lim_{n\to\infty} \frac{1}{n} \ln \vert F_n(z) \vert, \quad z \in V .
\]
\end{proof}

The above theorem  leads to an algorithm for drawing zero attractors.   
Locally,  the boundary of a  phase is a portion of an
integral curve of the  differential
equation:
\begin{align*}
\frac{dy}{dx}  
&  =
\frac
{\Re
\lbrack
L_{k}^{\prime}(x+iy)-L_{\ell}^{\prime}(x+iy)]}
{\Im
\lbrack L_{k}^{\prime}(x+iy)-L_{\ell}^{\prime}(x+iy)]},\label{eq:diff_eqn}
\end{align*}
where its  initial
condition is found by solving the appropriate single non-linear equation with a
specified radius or angle; typically on the unit circle, say $\Re L_k(e^{it}) = \Re L_\ell(e^{it})$.

For chromatic polynomials, the analogue of the phases is the equimodular set \cite{biggs}
which can also be described by means of  integral curves but of a single differential equation.

\section{A New Special Function: The Root Dilogarithm \label{section:fk}}

 The determination of the phases for the families of partition polynomials in this paper rests on finding the bounds among
 the family of functions given through the dilogarithm. 
 
 \begin{definition}\label{def:root_dilog}
 The root dilogarithms are the functions  given by
\begin{equation} \label{eq:f_k}
f_{k}(z)=
\frac{1}{k}\Re\left[ \sqrt{{\mathrm{Li}}_{2}(z^{k})}\right] 
\end{equation}
where the square root is chosen as nonnegative and $k$ is a positive integer.
\end{definition}

In general, the study of $f_k(z)$ is a lengthy tangent which we leave to the appendix.  
We will state here what we need and refer the reader to the appendix for their proofs. These results come in two kinds.  
First will be the calculus of $f_1(z)$ on radial lines and circles.

\begin{proposition}\label{f1calculus}
The function $t\to f_1(it)$ is a positive, increasing, concave function of $t\in (0,1).$
\end{proposition}

On circles, $f_k(re^{it})$ behaves in a cosine like fashion.  

\begin{proposition}
\label{prop:decreasing}
For a fixed value of $r\in (0,1],$ the function 
$t \mapsto f_{1}(re^{it})$ is decreasing on $[0,\pi ]$.
\end{proposition}

Second are root dilogarithm dominance facts that we will use extensively throughout our examples.

\begin{theorem} \label{fkdomanancetheorem}
For $0 < \vert z \vert \leq 1$, 
\begin{enumerate}
\item 
$ f_k(z)\le f_k(\vert z \vert)\le f_2(\vert z \vert) < f_1(z),$  $k\geq 2,$  $\vert\arg z\vert \le \pi/3.$
\item 
$ f_k(z)\le f_k(\vert z \vert) < f_1(z),$  $k\geq 3,$  $\vert\arg z\vert \le \pi/2$,
\item 
$ f_k(z) < f_1(z)$, $k \geq 2$, $0 \leq  \arg z \leq \pi/2$,
\item
$ f_k(z) < \max[  f_1(z), f_2(z), f_3(z)],$   $  k \geq 4$, $\pi/2 \leq \arg z \leq \pi$,
\end{enumerate}
\end{theorem}

\section{Partition Polynomials}

\subsection{Partition Polynomials Weighted by Number of Parts}\label{section:original_polynomials}

Consider the constant exponent sequence $\{a_m\}$ where $a_m=1$ for all $m$ so there is no restriction on the
parts of a partition.
For $k \in {\mathbb Z}^+$ and $t \in \widehat{ {\mathbb Z}_k}$,
 the Dirichlet series $D_{t,k}(s)$ are given by
\[
D_{t,k}(s) = \sum_{m=1}^\infty \frac{   e_k( mt) }{ m^s}
=
\frac{1}{k^s} \, \sum_{r=1}^k e_k( r t)  \zeta( s, r/k) 
=
F( t/k,s)
\]
where $F(\lambda,s)$ is the periodic zeta function
\[
F(\lambda,s) = \sum_{m=1}^\infty \frac{ e( m \lambda)}{ m^s} .
\]
$F(\lambda,s)$ is periodic in $\lambda$ with period 1. When $\lambda$ is not an integer,
then $F(\lambda,s)$ is an entire function of $s$; otherwise, $F(\lambda,s)$ reduces to the Riemann
zeta function.  

\begin{remark}
In this example $D_{t,k}(s)$ reduces to the usual  Dirichlet $L$-function
relative to a  character of ${\mathbb Z}_k$ when $(t,k)=1$.
\end{remark}

Since $\zeta(0, r/k) =1/2 - r/k$, for $1 \leq r < k$,
\[
D_{t,k}(0)=  \sum_{r=1}^k \, (\tfrac{1}{2} - \tfrac{r}{k} ) \, e_k( tr ) 
=
- \frac{1}{k} \sum_{r=1}^k r e_k( t r) ,
 \quad 1 \leq t < k ,
\]
while $D_{k,k}(0)=-1/2$.
Since $F(s ,\lambda)$ has a pole at $s=1$ when  $\lambda$ is an integer, we find
\[
Res( D_{t,k},s=1) =  Res( \zeta(s), s=1) \, \delta_{t,k} =    \delta_{t,k} .
\]
Hence its
  $\widehat{{\mathbb Z}_k}$ Fourier transform is
\[
c_k(j)
=
\frac{1}{k} , \quad 1\leq   j \leq k .
\]
For $h$  relatively prime to $k$  such that $1 \leq h < k$,
 we find  an explicit form for $L_{h,k}(z)$:
\begin{align*}
L_{h,k}(z)^2
&=
\sum_{j=1}^k  c_k(j) Li_2(   e_k(j) z)
=
\frac{1}{k} \sum_{ j=1}^k  Li_2(   e_k(j) z)
=
\frac{1}{k^2} Li_2(z^k ) ;
\end{align*}
in particular,  the functions $L_{h,k}(z)$ are independent of the choice of $h$ so we write $L_k(z)$ for them. 
Furthermore, each  $L_k(z)$ is analytic on the unit disk except for branch cuts at the rays
$\arg z = \pi j/k,$ $0 \leq j <k$.

For small values of $k$, we write out the values of $b_k$:
\[
b_1(1)=-1/2 ; \, b_2(1) = (-1)^n /2 ; \,
b_3(1)= 1/6,  b_3(2)=-1/6,  b_3(3)=-1/2 .
\]
Hence we find that
\begin{align*}
\omega_{1,1,n}(z)&= 
\sqrt{1-z}, \quad \omega_{1,2,n}(z)= (-1)^n \sqrt{1-z},
\\
\omega_{1,3,n}(z)
&=
e_3(-n) \,
\frac{ (1- e_3(2) z)^{1/6} (1-z)^{1/2}}{  ( 1- e_3(1) z) )^{1/6}} ,
\\
\omega_{2,3,n}(z)
&=
e_3(-2n) \,
\frac{ (1- e_3(1) z)^{1/6} (1-z)^{1/2}}{  ( 1- e_3(2) z) )^{1/6}} ,
\\
\end{align*}
We can verify directly that $\omega_{1,3,n}(z) \neq  - \omega_{2,3,n}(z)$ for $\vert z \vert<1$. 
Using Theorem \ref{fkdomanancetheorem}, we obtain that there are only three phases $R(1),$ $R(2),$ and $R(3).$  These phases can be more concretely written down as
\begin{align*}
R(1) &= \left\{z\in \mathbb{D}: \Re \sqrt{Li_2(z)}>\frac{1}{2} \Re \sqrt{Li_2(z^2)},\frac{1}{3} \Re \sqrt{Li_2(z^3)}\right\}, \\
R(2) &= \left\{z\in \mathbb{D}: \frac{1}{2}\Re \sqrt{Li_2(z^2)}>\frac{1}{2} \Re \sqrt{Li_2(z)},\frac{1}{3} \Re \sqrt{Li_2(z^3)}\right\}, \\
R(3) &= \left\{z\in \mathbb{D}: \frac{1}{3}\Re \sqrt{Li_2(z^3)}>\frac{1}{2} \Re \sqrt{Li_2(z)}, \Re \sqrt{Li_2(z)}\right\}.
\end{align*}
Using Theorem \ref{theorem:asymptotic_form} we record the following proposition.
 \begin{proposition}\label{prop:partitionpolyphase}
The zero attractor for the partition  polynomials is
  \[
  S^1 \cup \partial R(1) \cup \partial R(2) \cup \partial R(3) .
  \]
  \end{proposition}

 \begin{figure}\label{fig3}
 \begin{center}
        \includegraphics[height=4.5cm,width=4.5cm]{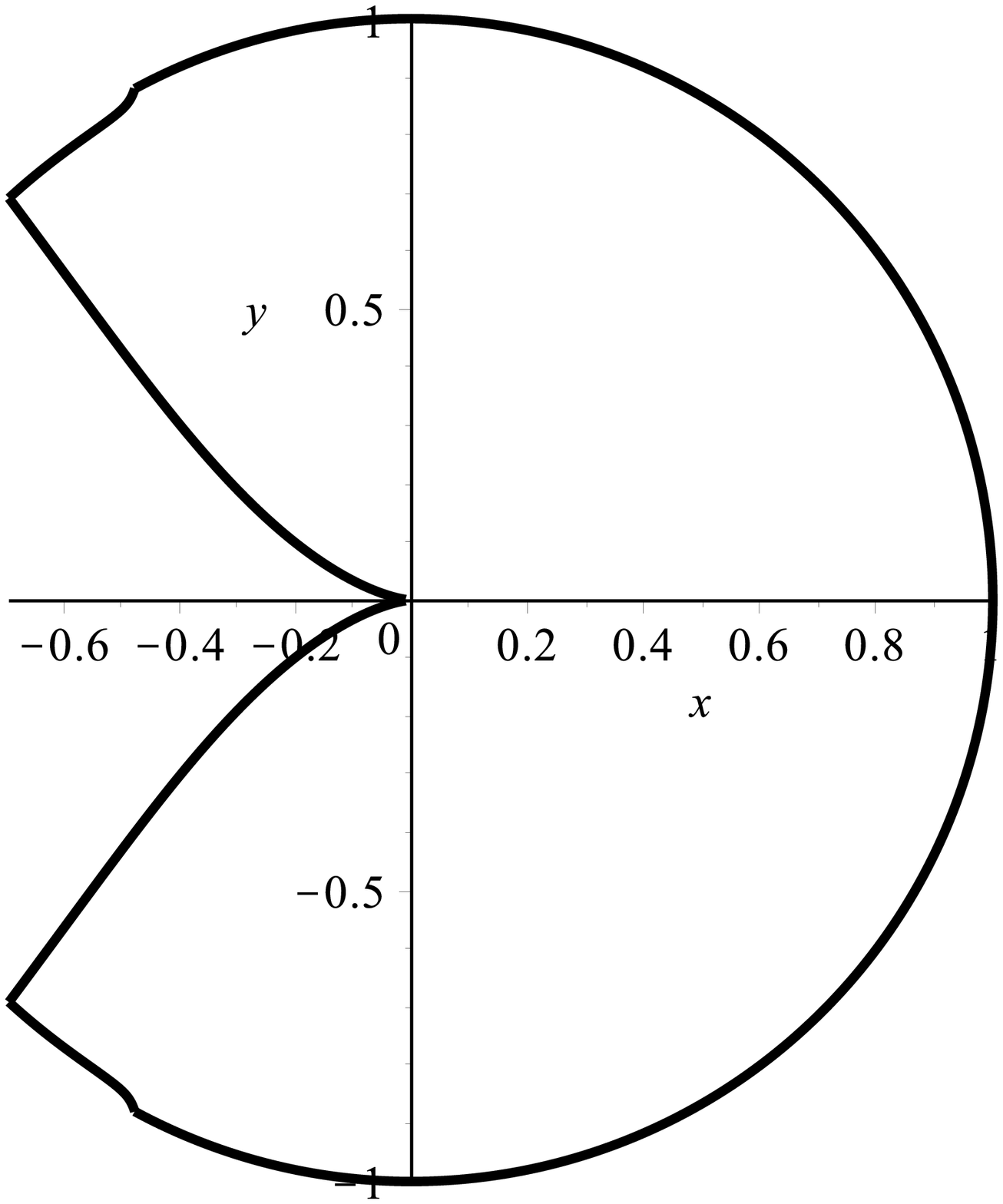} 
        \,
\includegraphics[height=4.50cm,width=4.50cm]{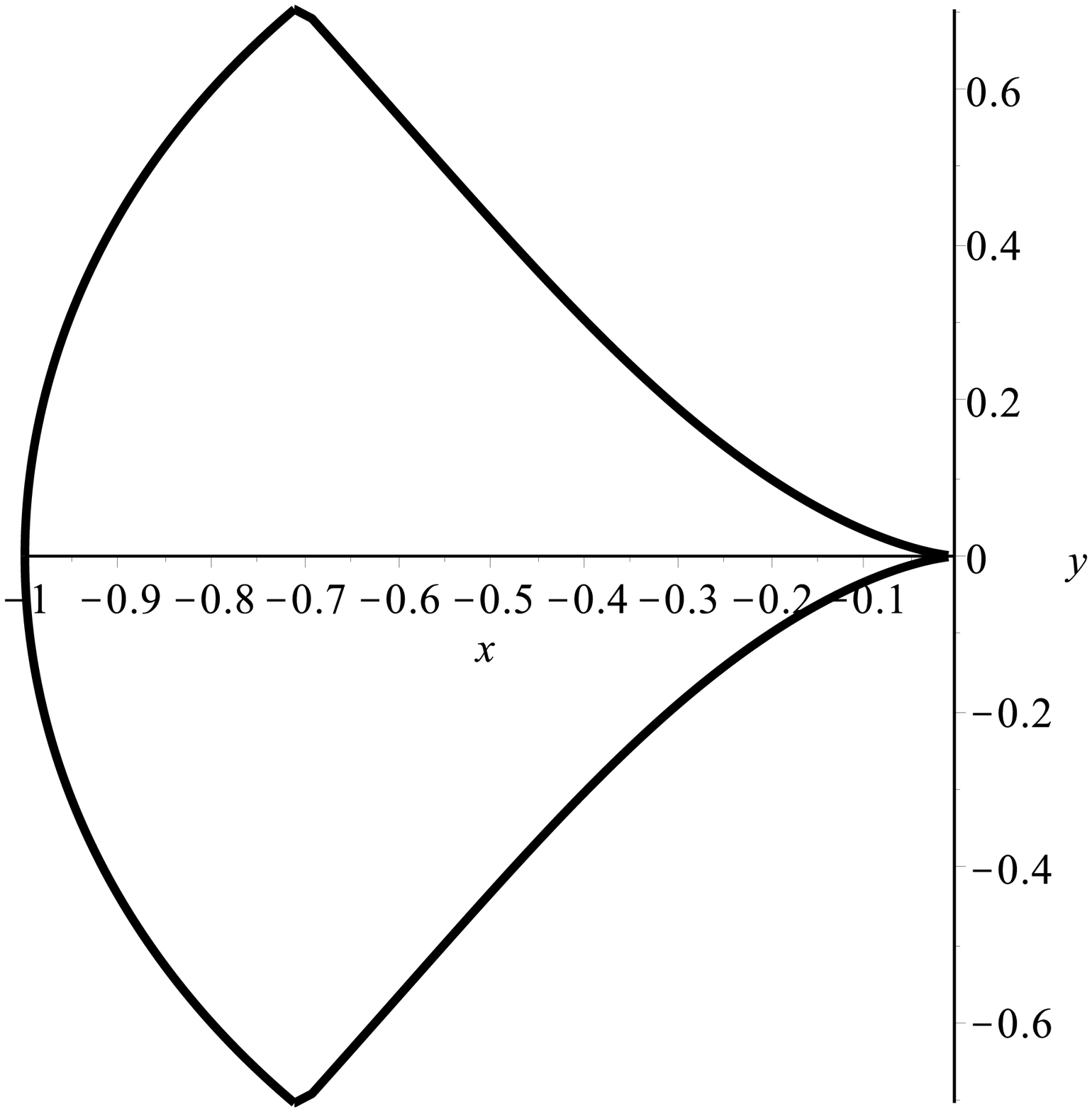} 
\qquad
\includegraphics[height=4.5cm,width=4.5cm]{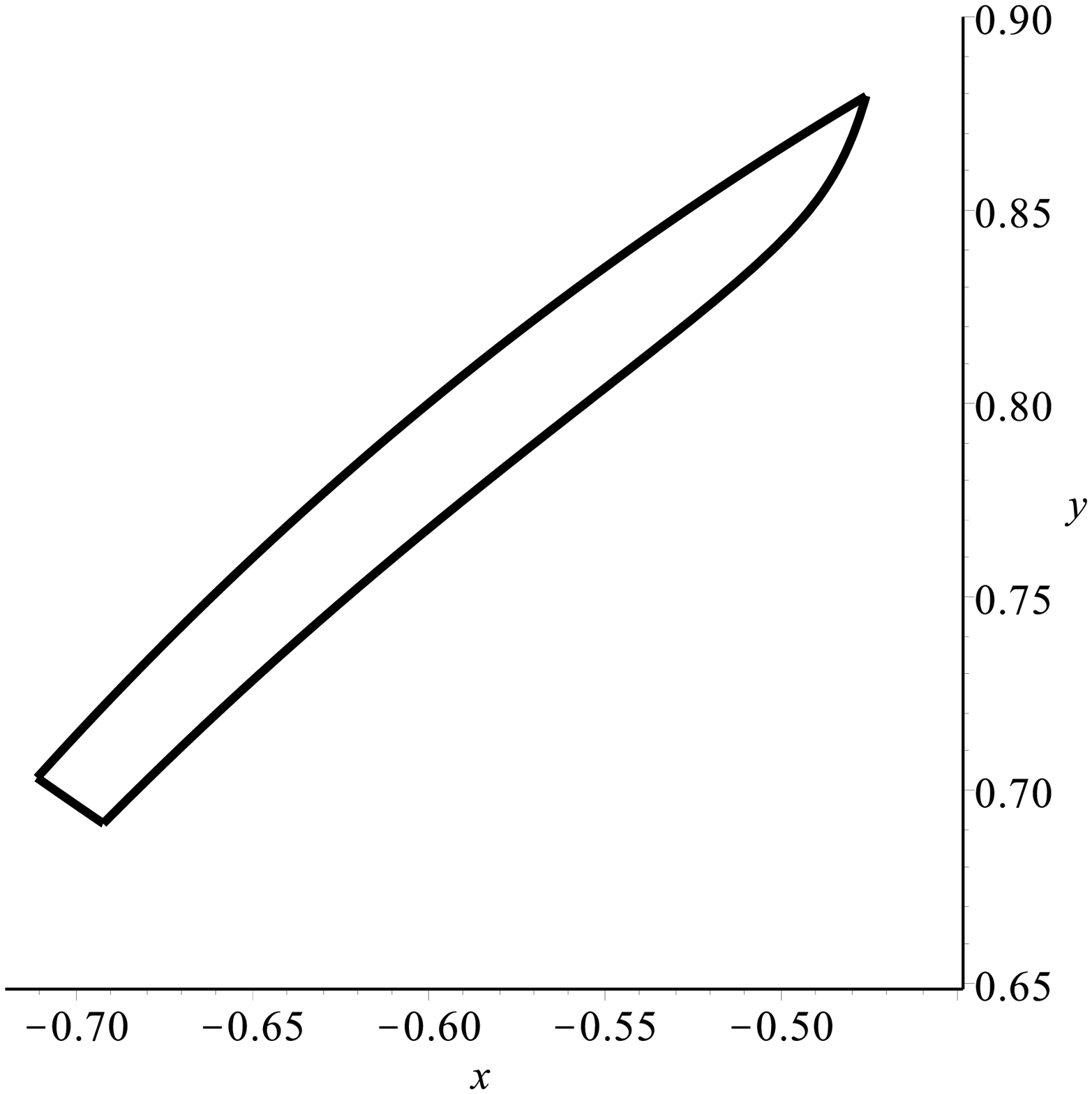} 
\end{center}
\caption{
Phases for weighted partition polynomials: (a) Region 1
(b) Region 2, (c) Region 3 in upper half plane.
}
\end{figure}

   \begin{figure}\label{fig4}
   \begin{center}
        \includegraphics[height=5.18cm,width=5.18cm]{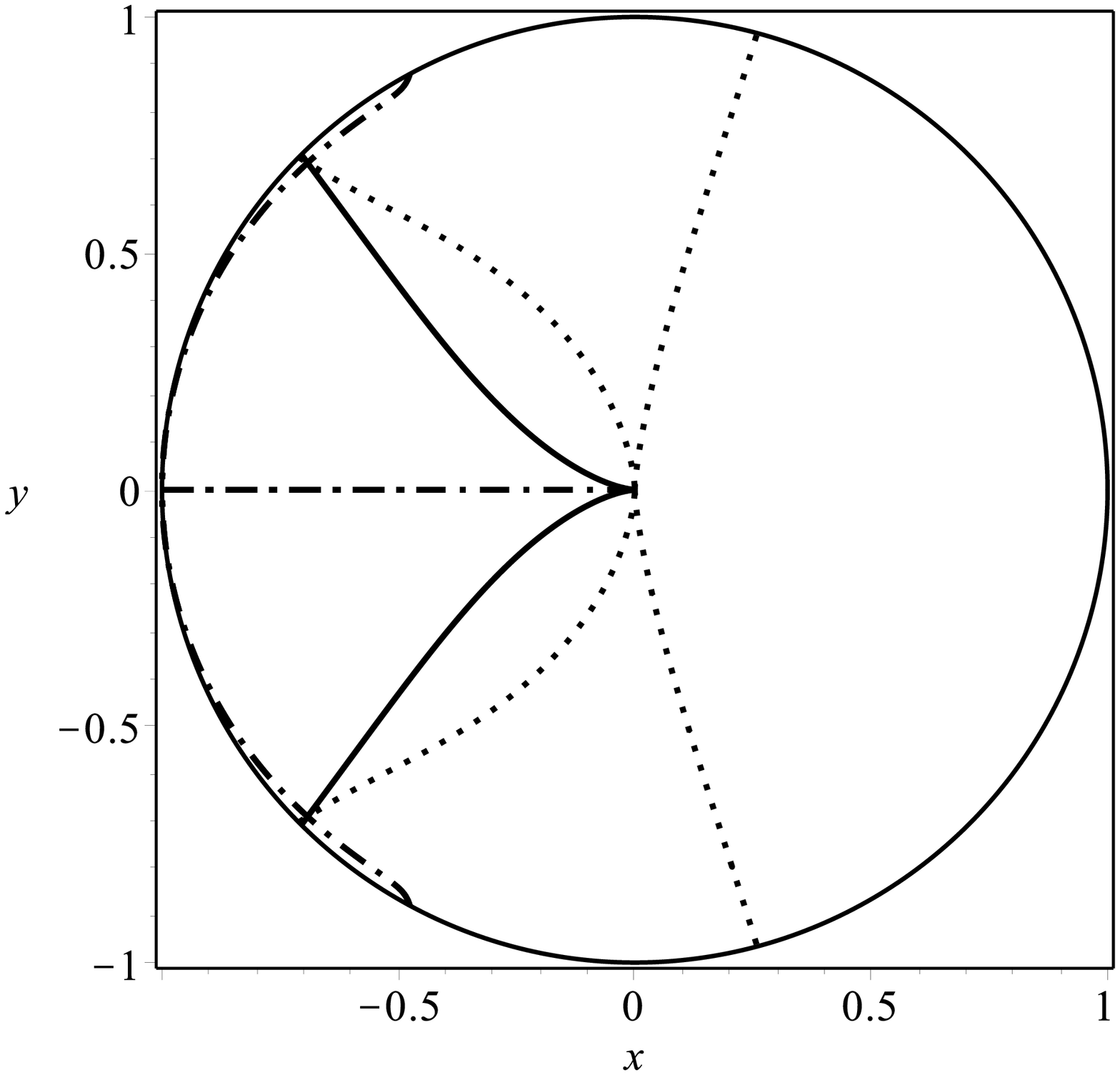} 
        \qquad
\includegraphics[height=5.05cm,width=5.05cm]{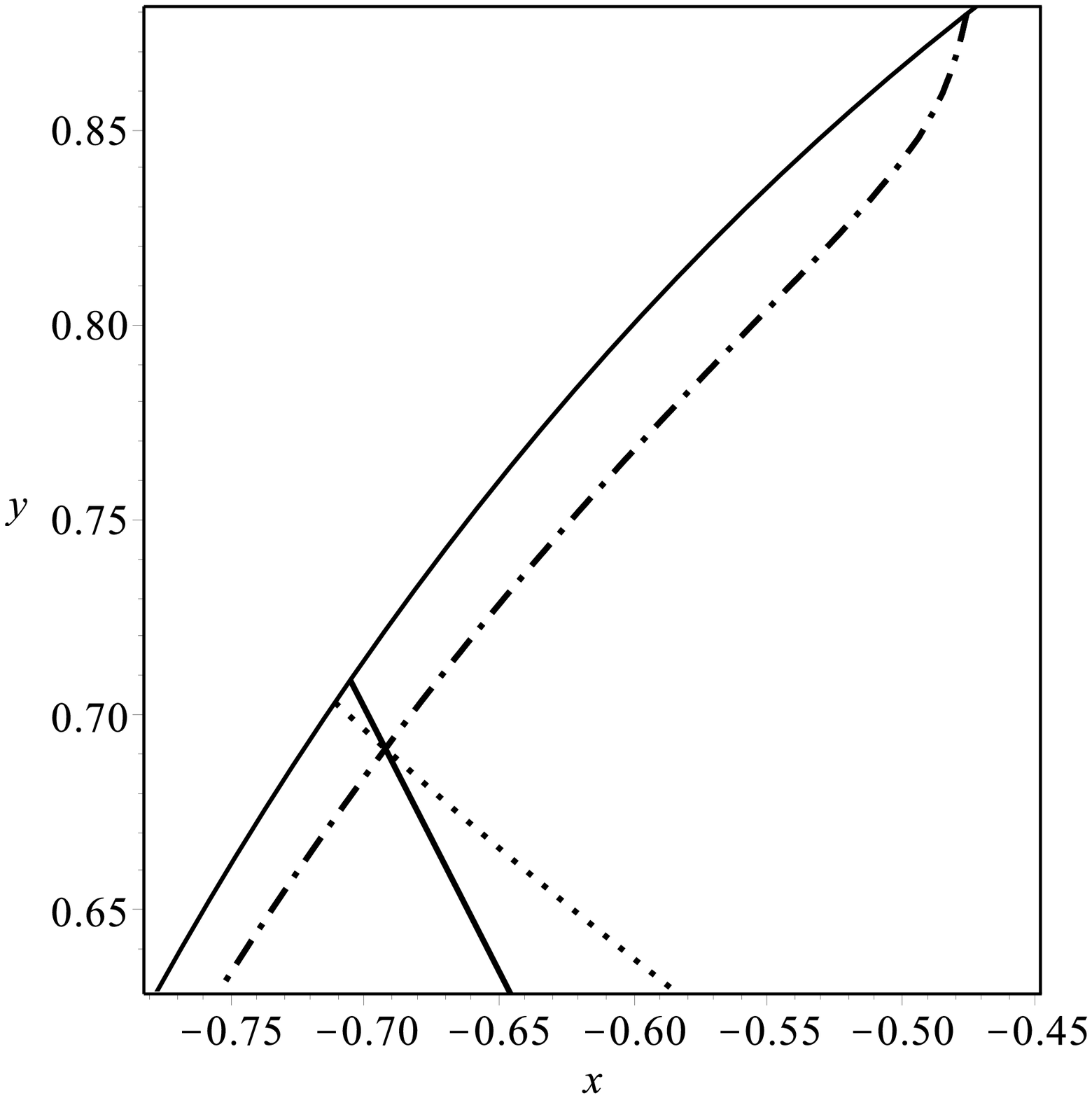} 
\end{center}
\caption{
(a) Level sets $f_1(z)=f_2(z)$ (solid line), $f_2(z)=f_3(z)$ (dotted line), $f_1(z)=f_3(z)$ (dash-dot lines) inside the unit disk, (b) Closeup near the triple point
}
\end{figure}

\begin{remark} 
The appendix shows that there are at most three phases; in fact, all three do occur. In the right-half plane,
including the imaginary axis, $f_2(z), f_3(z) < f_1(z)$. In the second quadrant, consider 
$f_1( r e^{it})$ which is decreasing to $0$ while  $f_2( r e^{it})$ is increasing from $0$, $t \in [\pi/2,\pi]$. 
Hence, for each $0<r \leq 1$, there is a unique value $t_{12}(r)$ such that
$f_1( r e^{ i t_{12}(r)}) = f_2( r e^{ i t_{12}(r)})$.  Similar reasoning shows 
$f_2( r e^{it}) = f_3( r^{it})$ has at most one solution for $2 \pi/3 \leq t \leq \pi$; in fact,
$f_3( r e^{ i 5 \pi /6}) < f_2( r e^{ i 5 \pi /6})$ on $(0,1]$.

The phase $R(1)$ includes the unit disk that lies in the right-half plane.
The phases $R(2)$ and $R(3)$ lie in the left-hand plane since $f_1(z) > f_k(z)$,
for $k \geq 2$, $0<\vert z \vert \leq 1$, and $\Re z \geq 0$.   For fixed $r >0$,
$f_1( r^{it})$ is decreasing to $0$ while $f_2( e^{ it})$ is increasing from $0$ on $[\pi/2,\pi]$.
So $f_1(z) = f_2(z)$ determines a level set that connects $0$ to a unique point on the unit circle.
See Figure \ref{fig3}.
\end{remark}

\subsection{Partition Polynomials Whose Parts Satisfy a Linear Congruence}\label{section:linear_congruence}

For a fixed  positive integer $p \geq 2$,
consider the exponent sequence $\{a_m\}$ whose nonzero entries satisfy $ m \equiv 1 \, {\rm mod} \,  p$. 
The  corresponding partition polynomials  count partitions all of whose parts satisfy the congruence
$1 \mod p$.
 For $k\in {\mathbb Z}^+$ and  $t \in \widehat{ {\mathbb Z}_k}$,
let
    \[
    D_{t,k}(s) 
    =
    \sum_{m=1}^\infty \frac{a_m e_k(tm)}{m^s}
    =
     \sum_{\ell=0}^\infty \frac{ e_k(t(p \ell+1))}{(p \ell+1)^s}
     =
      \frac{e_k(t)}{p^s} L(p t/k, s,1/p)
    \] 
     where $L(\lambda,s,\nu)$ is the Lerch zeta function. 
     We observe that  if $k \mid pt$ then $D_{t,k}(s)$ reduces to 
     $e_k(t) p^{-s} \zeta(s, 1/p)$. Since $\zeta(0, \alpha) = 1/2-\alpha$ and
     \[
L(\lambda,0, \nu  ) =  \frac{1}{1 - \exp(2 \pi i \lambda )} ,
\quad \lambda  \notin {\mathbb Z} .
\]
 we have for $1 \leq t \leq k$
     \[
     D_{t,k}(0) = 
     \begin{cases}
    e_k(t) ( 1/2-1/p),& k \mid pt,
     \\
      \frac{e_k(t)}{1 - e_k(pt)  } ,&  k  \nmid pt
     \end{cases}
     \]
     In the special case for $p=2$ with  $1\leq t < k$ and $t$ is relatively prime to $k$
     \[
     D_{t,k}(0)=  \frac{ i}{2 \sin( 2 \pi t/k)} 
     \]
     and, with $p=2$,   $D_{k,k}(0)=0$.

    Recall that $L( \lambda,s, \nu)$ is an entire function of $s$ if and only if $\lambda \notin {\mathbb Z}$. 
    When $\lambda$ is
    an integer, $L( \lambda,s, \nu)$ is a mereomorphic function of $s$ with a unique singularity at $s=1$ which is a simple pole with residue $1$.
    Hence $D_{t,k}(s)$ has a singularity if and only if $k \mid pt$. If this is case, then the singularity is a simple pole at $s=1$ with residue
\[
Res( D_{t,k}(s), s=1)
=
\frac{1}{p} e_k(t) .
\]
\begin{lemma}
For $p \geq 3$, the finite Fourier transform $b_k(t)$ of $ t \mapsto D_{t,p}(0)$ is
\[
b_k(t)
=
\begin{cases}
\frac{p/2-1}{p},& t=1, 2  \mid p \textrm{ and }  2 \mid p 
\\
\frac{p-2 }{ 2p} , & t=1, 2 \nmid p
\\
0,& 2 \leq t \leq p 
\end{cases}
\]

\end{lemma}

\begin{example}
For the special case: $p=2$ with small values of $k$,
   we record the values of the finite Fourier transform $b_k$:
    \begin{align*}
     b_{3}(1)&= - b_{3}(2)= 1/3, \quad b_{3}(3)=0
    \\
    b_{4}(1) &=b_{4}(3)= 1/4, \quad b_{4}(2)=b_{4}(4)=0,
    \\
    b_{6}(1) &= - b_{6}(5) = 1/3, \quad b_{6}(2)=b_{6}(3)=b_{6}(4)=b_{6}(6)=0 .
    \end{align*}
  This allows us to write out the explicit forms for the corresponding $\omega_{h,k,n}(z)$ functions:
\begin{align*}
\omega_{1,1,n}(z)
 &=
 1,
 \quad \omega_{1,2,n}(z) 
 =
 (-1)^n,\quad 
 \\
 \omega_{1,4,n}(z) 
 &=
 i^{-n}\sqrt[4]{\frac{z-i}{z+i}}, \quad \omega_{3,4,n}(z) 
 =
 i^{n}\sqrt[4]{\frac{z+i}{ z-i}}
 \end{align*}
 Further we can check directly that $\omega_{1,4,n}(z)  = - \omega_{3,4,n}(z) $ only if $z=0$.
 \end{example}
 
 \begin{example}
For  $p=3$,  we have
\begin{align*}
b_{1}(1) &= - 1/2, \quad b_{2}(1)=1/3, b_{2}(2)=-1/6
\\
b_{6}(1) &= 1/3, b_{6}(4)=-1/6, \quad b_{6}(2)=b_{6}(3) = b_{6}(5) = b_{6}(6)=0
\end{align*}
 \end{example}

    \begin{lemma}
If $(k,p)=m>1$, then $c_k(j)=
\begin{cases}
\frac{(k,p)}{kp},& j \equiv 1 \mod (k,p),
\\
0,& j \not\equiv 1 \mod (k,p)
\end{cases}$
\end{lemma}
\begin{proof}
Let $1 \leq t \leq k$. Then $ k \mid pt $ if and only if $k_0 \mid p_0t$ where $k_0=k/(k,p)$ and $p_0 = p /(k,p)$. In other words,
$k \vert \mid pt$ if and only if $p_0 t \in k_0 {\mathbb Z}$. In our case, $t= v k_0$, $1 \leq v \leq  (k,p)$. Consider
\begin{align*}
c_k(j)
&=
\frac{1}{k} \sum_{t=1}^k e_k( -jt) \, Res( D_{t,k}(s),s=1)
=
\frac{1}{k} \sum_{t=1}^k e_k( -jt) \, \delta_{  k \mid pt } \frac{ e_k(t)}{p}
\\
&=
\frac{1}{k} \sum_{t=1}^k e_k( -jt) \, \delta_{  k_0 \mid p_0 t } \frac{ e_k(t)}{p}
=
\frac{1}{k} \sum_{v=1}^{ (k,p) } e_k( -j v k_0) \, \frac{ e_k( vk_0)}{ p}
\\
&=
\frac{1}{kp} \sum_{v=1}^{(k,p)} e_k ( v k_0 (1-j)) 
=
\frac{1}{kp} \sum_{v=1}^{ (k,p)} e_{ (k,p)}( v (1-j) )
\\
&=
\begin{cases}
\frac{ ( k,p) }{pk}  , & j \equiv 1 \mod (k,p)
\\
0, &  j \not\equiv 1 \mod (k,p)
\end{cases}
\end{align*}
\end{proof}

We use this Kubert-type identity in the next lemma:
  \[
   \sum_{m=1}^{k} Li_2(  z e_k( m )) = \frac{1}{k} Li_2( z^k ) .
   \]

\begin{lemma} 
Let $1\leq h < k$ and $(h,k)=1$. Then
   \[
   \displaystyle   L_{h,k}(z)^2 
   =  \frac{  (k,p)^2 }{ k^2 p}
Li_2  \left(   \,  e_{  (k,p)} (h  )  \,  z^{ k/(k,p)}  \right) .
\]
   \end{lemma}
\begin{proof}
Consider the following expansions
\begin{align*}
&L_{h,k}(z)^2
=
\sum_{j=1}^k c_k(j) Li_2 \left( e_k( j h) z  \right) 
=
\sum_{ k \mid (j-1)} \frac{ (k,p)}{kp} Li_2 \left ( e_k( j h) z \right) 
\\
&=
\sum_{v=1}^{ k/( k,p)}  \frac{(k,p)}{kp} Li_2 \left( e_k(   ( v(k,p)+1) h )z \right)
\\
&=
 \frac{(k,p)}{kp}
\sum_{v=1}^{ k/( k,p)}  Li_2 \left ( e_{k/(k,p) } (    v  )  \,  [e_k(h) z]  \right)
\\
&=
\frac{ (k,p)}{ k p} \, \frac{1}{k/ (k, p)}
Li_2 \left (   [ e_k(h)z]^{k/ (k,p)} \right) 
\\
&=
\frac{ (k,p)}{ k p} \, \frac{1}{k/ (k, p)}
Li_2 \left(   [ e_k(h k/(k,p) ) ]  \,  z^{ k/(k,p)}  \right) 
\\
&=
\frac{  (k,p)^2 }{ k^2 p}
Li_2  \left(   \,  e_{  (k,p)} (h  )  \,  z^{ k/(k,p)} \right)   .
\end{align*}
\end{proof}

\begin{remark}\label{arethmetlkcomputation}
For polynomials counting partitions whose parts are congruent to $a$ modulo $p$ with $(a,p)=1$, then
\[
L_{h,k}(z)^2 = 
\frac{  (k,p)^2 }{ k^2 p}
Li_2  \left(   \,  e_{  (k,p)} (h a  )  \,  z^{ k/(k,p)} \right)   .
\]

\end{remark}

    For   $p \geq 2$,  assume that  the only nonzero entries of the
    exponent sequence $\{ a_m\}$ satisfy $m \equiv 1 \,  {\rm mod} \, p$.
       Calculating directly from the generating function $G(z,q)$, we get
       \[
F_\ell( e^{ 2 \pi i / p} z)  =  e^{2 \pi i \ell /p}  F_\ell(z)
\]
Since the polynomial $F_\ell(z)$ has real coefficients, $\overline{ F_\ell(z)} = F_\ell( \overline{z})$ as well. So their  moduli $\vert F_\ell(z)\vert$ are invariant under
the action of the dihedral group generated by the rotation by angle $2\pi/ p$ and reflection about the real axis.  
Note the choice of $1 \mod p$ is done for simplicity and a similar analysis can be done for $a \mod p$ together with the
requriement  $a_1=1$ although the computations are needlessly lengthy.  
Nonetheless, for clarity for  the reminder of this section, we will still use the two variable notation for phases
\[
R(h,k) = \{z\in \mathbb{D}\setminus 0: \Re L_{h,k}(z) \ \textrm{is maximal}\}.
\]
 with $R(p,p)=R(1,1)=R(0,p)$ for notational uniformity.  
The phase functions connect to the root dilogarithm function by
\begin{equation}\label{rsqarethemetic}
\Re L_{h,k}(z) = \frac{1}{\sqrt{p}}f_{k/(k,p)}(e_k(h)z).
\end{equation}
We need an elementary number theory lemma to handle the cyclic nature of 
$\Re L_{h,k}(z).$

\begin{lemma} \label{MsicNumTheoreyLemma1}For every $k,j\in \mathbb{N},$ $k/(k,j)=1$ if and only if $k\mid j.$ \end{lemma}
\begin{proof}
Suppose that $k\mid j$ and of course $k\mid k.$  By the definition of greatest common divisor, $k\mid (k,j).$  But $k\le j$ so $k/(k,j)=1.$ If $k/(k,j)=1$ then $k\mid (k,j).$  Therefore by the definition of greatest common divisor, $k\mid j.$
\end{proof}

The next lemma is a consequence of Theorem \ref{fkdomanancetheorem} part (1).

\begin{lemma} \label{lemma:odd_max}
We can reduce the set of phase functions for every $z\in \mathbb{D},$ 
$j \geq 3,$ 
\[
\max
\{ 
\Re L_{h,k}(z)  : k\mid p,    (k,p)=1 \}
>
\max \{\Re L_{h,k}(z) :  k\nmid p \}
\] 
In terms of phases we have
\[
\overline{\mathbb{D}} = \bigcup_{h\in \mathbb{Z}_p}\overline{R(h,p)}.
\]
 \end{lemma}
 \begin{proof}
If $p$ is a prime, we observe, using Lemma \ref{MsicNumTheoreyLemma1},  that
if $p/(p,k)=1$ then $k\mid p$ and either $k=1$ or $k=p.$
In the case $k=1$, then $h=1$ since $(h,k)=1$; otherwise, $k=p$ and so  $h\in \mathbb{Z}_p^\times.$
Using the definition of $\Re L_{h,k}(z)$ in Equation \ref{rsqarethemetic}
and  applying Proposition \ref{prop:decreasing}, we obtain
\begin{align*}
\max \{ \Re L_{h,k}(z) : k\mid p,   (k,h)=1 \}
&=
 \max \{ \frac{1}{\sqrt{p}}f_{1}\left(e_p(h)z\right) : 1\leq h \leq  p \} 
 \\
&=
\frac{1}{\sqrt{p}}f_{1}\left(e_p(h^*)z\right).
\end{align*}
where  $h^*$ is chosen as the minimizer of  $\left|\arg \left(e_p(h^*)z\right) \right|$ over $h\in \mathbb{Z}_p$. 
By Theorem \ref{fkdomanancetheorem} part (1) and the Equation \ref{rsqarethemetic},
we next find that  if $ \left|\arg \left(e_p(h^*)z\right) \right|\le \pi/3$ and $k \nmid p$ then
\[
\Re L_{h,k}(z)
 = 
 \frac{1}{\sqrt{p}}f_{k/(k,p)}\left(e_k(h)z\right) \le \frac{1}{\sqrt{p}}f_2(|z|)
  \leq
   \frac{1}{\sqrt{p}}f_1(e_p(h^*)z).
\]
Hence it is enough to show there exists at least one $h\in \mathbb{Z}_p$
 such  that 
 $\left|\arg \left(e_p(h)z\right) \right| \leq {\pi }/{p} \leq  {\pi}/{3}.$
 But this is a classic consequence of the pigeonhole principle.  
\end{proof}

\begin{theorem}
The phase $R(h,p)$ is an open angular wedge of angle $2\pi /p$ of radius 1 centered at $\arg z =2  h \pi/p.$ 
\end{theorem}
\begin{proof}
By Lemma  \ref{lemma:odd_max}, we already know that 
\[
\overline{\mathbb{D}} = \bigcup_{h\in \mathbb{Z}_p} \overline{R(h,p)}
\]
 We also have a symmetry that $f_1(e_p(h)z) = f_1(e_p(h-1)z e_p(1)z)$ implying that if 
 $z\in R(h,p)$ if and only if $e_p(1)z \in R(h-1,p).$  Set wise we have
\[
R(n,p)=e_p(1)R(n-1,p). 
\] 

Hence we need only to show $R(1,1)$ contains the wedge $\vert\arg z \vert< \pi/p$ and because phases are disjoint and symmetric we will be done.  But by Proposition \ref{prop:decreasing}, Equation \ref{rsqarethemetic} and the previous lemma we need only determine for 
$\vert\arg z\vert\leq \pi/p$ that $\vert\arg z\vert = \min_{1\le h \le p}\vert\arg\left( e_p(h)z\right)\vert.$

Suppose $\vert\arg e_p(n)\vert>\frac{3\pi}{2p}$ then by the triangle inequality 
$\vert\arg ze_p(n)\vert>\frac{\pi}{p}$ and by Proposition \ref{prop:decreasing}, 
\[
f_1(z)\geq f_1(e_{2p}(1))>f_1(ze_p(n)).
\] 
This leaves only two potential candidates which can dominate $f_1(z);$
namely,
 $f_1(e_p(-1)z)$ and $f_1(e_p(1)z).$  So long as
 \[
 \frac{\pi}{p}-\arg z>\arg z, \quad \arg z+ \frac{\pi}{p}>  - \arg (z),
 \] 
 then by Proposition \ref{prop:decreasing}, $f_1(z)$ dominates the both of them.  This concludes the proof.
\end{proof}

Like the other examples all we need now do is apply Theorem \ref{theorem:asymptotic_form}. 
Of course, the boundaries of wedges can be written as the solution to $z^p=-1$ and so we conclude the example. 
              
\begin{proposition}\label{prop:original_asymptotics}
Given $p >2$, the zero attractor for partition polynomials corresponding
 to partitions into parts congruent to $1\mod p$ is the unit circle and a set of $p$ spokes.
$$
S^1\cup \{z\in \mathbb{D}: z^{p}=-1\}
$$
\end{proposition}
       
       \subsection{Odd Partition Polynomials} 
When $p=2,$ we obtain a nondegenerate behavior for the odd parts polynomials.  
Using Remark \ref{arethmetlkcomputation}, we obtain
\begin{align*}
\sqrt{2}L_{h,2k-1}(z) = L_{2k-1}(z)
&=
\frac{1}{2k-1} \sqrt{ Li_2( z^{ 2k-1})},
\\
\sqrt{2}L_{h,2k}(z)= L_{2k}(z)
&=
\frac{1}{k} \sqrt{ Li_2( - z^k) } .
\end{align*}
Once again that we note the connection to the root dilogarithm function 
where $\Re L_k(z) = f_k(z)$ for $k$ odd and  $\Re L_{2k}(z) = f_{1}(-z^{k})/k$ for even indices.

 \begin{theorem}  The only nonempty phases are $R(1),$ $R(2),$ and $R(4).$ \end{theorem}
 \begin{proof} Inspection of $L_{k}(z)$ near the points $z=1,-1,i$ proves that $R(1),$ $R(2),$ and $R(4)$ are nonempty.  For $k>4,$ we note that $\Re L_k(z) \le f_3(\vert z \vert)$ which by Theorem \ref{fkdomanancetheorem} must be bounded by either $\Re L_1(z)=f_1(z)$ or $\Re L_{2}(z)= f_1(-z)$ depending on whether
  $\Re z>0$ or $\Re z<0.$    Hence $R(k)$ is empty.

The case of $k=3$ follows by the same Lemma since
\[
\Re L_{3}(z) = f_3(z) \leq f_3(\vert z \vert) \leq \max\left( \Re L_1(z),\Re L_2(z)\right).
\]
 \end{proof}
 
Using Theorem \ref{theorem:asymptotic_form} we record the following proposition.

 \begin{proposition}\label{prop:partitionpolyphase}
The zero attractor for the odd parts partition polynomials is given by
  \[
  S^1 \cup \partial R(1) \cup \partial R(2) \cup \partial R(4) .
  \]
  \end{proposition}
However, we can go a bit further and identify curves to the boundaries of $R(m).$

\begin{proposition}\label{positivitybeta}
On $(0,1]$,
there is a unique solution $\beta$ to $f_1(i r)= f_2(r)$; further,  $3/4< \beta <1$.  
\end{proposition}

\begin{theorem}\label{Oddpartsphases}  
There exists exactly three phases $R(1),$ $R(2),$ and $R(4).$
\begin{enumerate}
\item 
The boundary of $R(4)$ consists of $ \gamma$
$-\gamma,$ $\overline{\gamma},$ 
and $-\overline{\gamma}$
where $\gamma$ is the level set
$ \{ z :  \Re L_1(z) = \Re L_4(z);  \Re z , \Im z \geq 0\}$.
\item 
The boundary of $R(1)$ is $  i [0,\beta] \cup  i  [-\beta,0] \cup\gamma \cup \overline{\gamma}.$
\item 
The boundary of $R(2)$ is $i [0,\beta]\cup  i [-\beta,0]\cup-\gamma \cup -\overline{\gamma}.$
\end{enumerate}
\end{theorem}  
\begin{proof}
Since  $\Re L_1(-z)=\Re L_2(z)$ and  $ \Re L_4(z)=\Re L_4(-z)$,
 we need only classify the phases on the upper quarter disk 
 $\mathbb{D}^{++}.$ On this quarter disk, we have by Proposition \ref{prop:decreasing} that 
 $\Re L_1(z) > \Re L_2(z)$, $z \neq 0$. 
It remains to examine $\Re L_1(z)$ and $\Re L_4(z).$
Consider the function
 \[
 h(z)= \Re L_1(z)-\Re L_{4}(z)=f_1(z)-\frac{1}{2}f_1(-z^2).
 \]  
 If $z\in \mathbb{D}^{++}$ and $h(z)>0$ then $z\in R(1)$ and if $h(z)<0$ then $z\in R(4).$  
 But by Proposition \ref{prop:decreasing} $h(re^{it})$ is decreasing.

The theorem now follows from Proposition \ref{positivitybeta}, and the intermediate values theorem.
\end{proof}  
   
The more detailed reduction of the zero attractor now follows directly from Theorem \ref{theorem:asymptotic_form}.

\begin{theorem} \label{innerasym} 
          If $A$ is the zero attractor of $F_n(z)$ then \[
 A= S^1 \cup \gamma \cup -\gamma \cup \overline{\gamma} \cup - \overline{\gamma}\cup i[-\beta, \beta] .
  \]
   \end{theorem}

     \begin{figure}\label{fig1}
     \begin{center}
        \includegraphics[height=4.5cm,width=4.5cm]{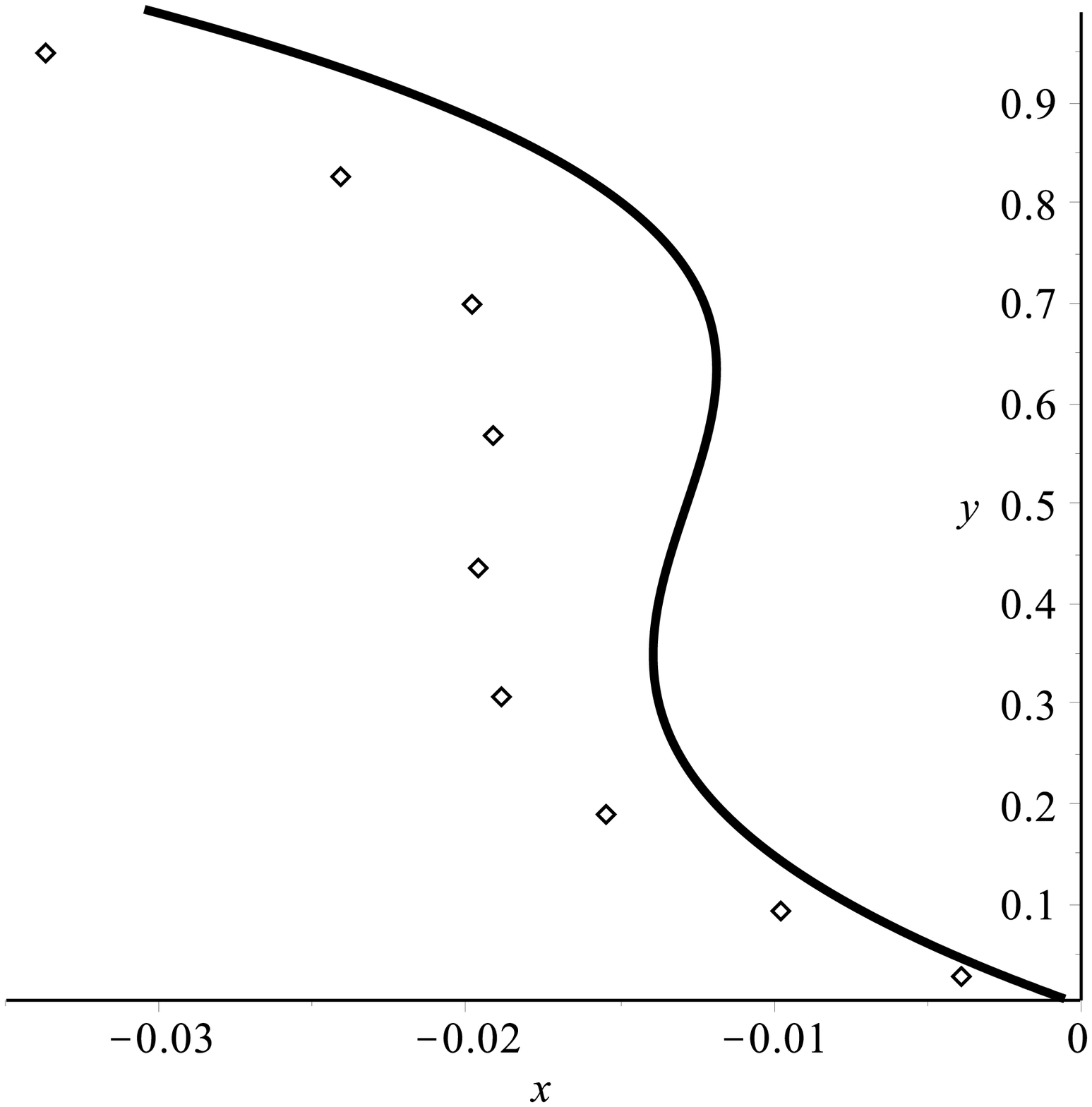} 
        \quad
\includegraphics[height=5.0cm,width=5.0cm]{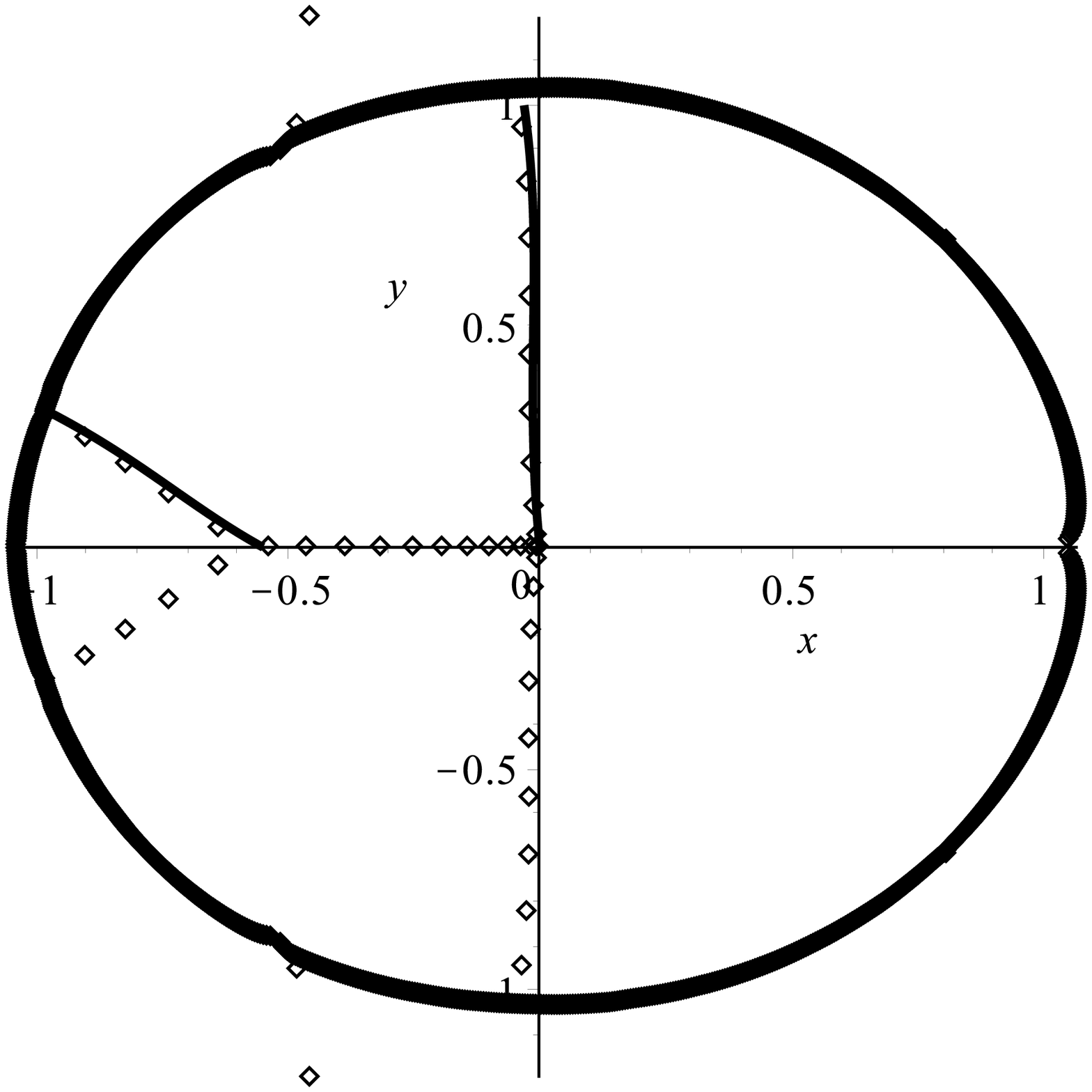} 
\qquad
\includegraphics[height=4.5cm,width=4.5cm]{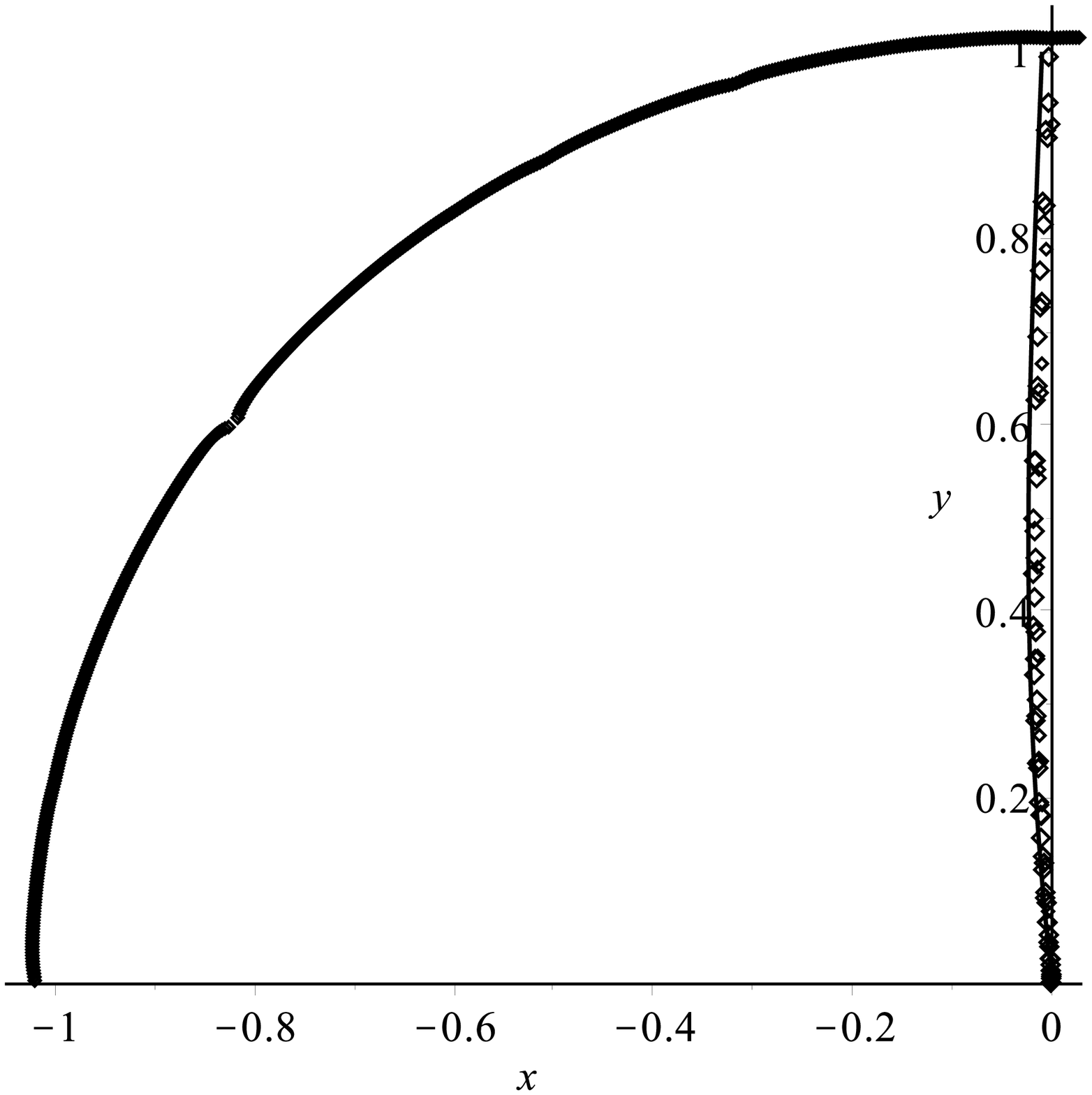} 
\end{center}
\caption{Closeups of  zero attractor inside unit disk:
(a) Closeup: parts are congruent to 0 or 2 modulo 3 \,
(b); parts are congruent to 0 or 2 modulo 3
\,
(c) parts are congruent to 1 or 4 modulo 5,
}
\end{figure}

  \begin{figure}\label{fig2}
       \begin{center}
        \includegraphics[height=4.5cm,width=4.5cm]{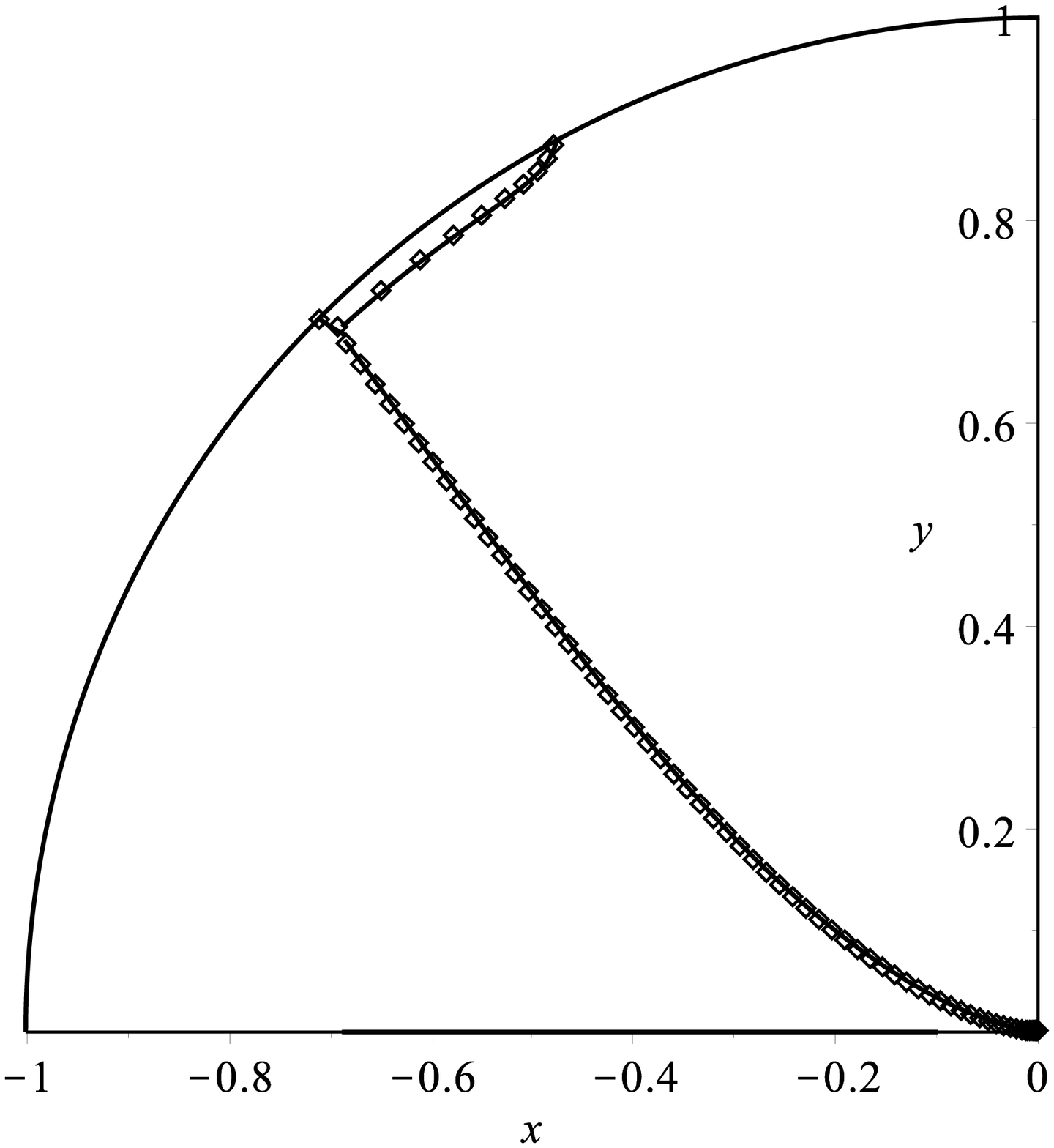} 
        \quad
\includegraphics[height=5.0cm,width=5.0cm]{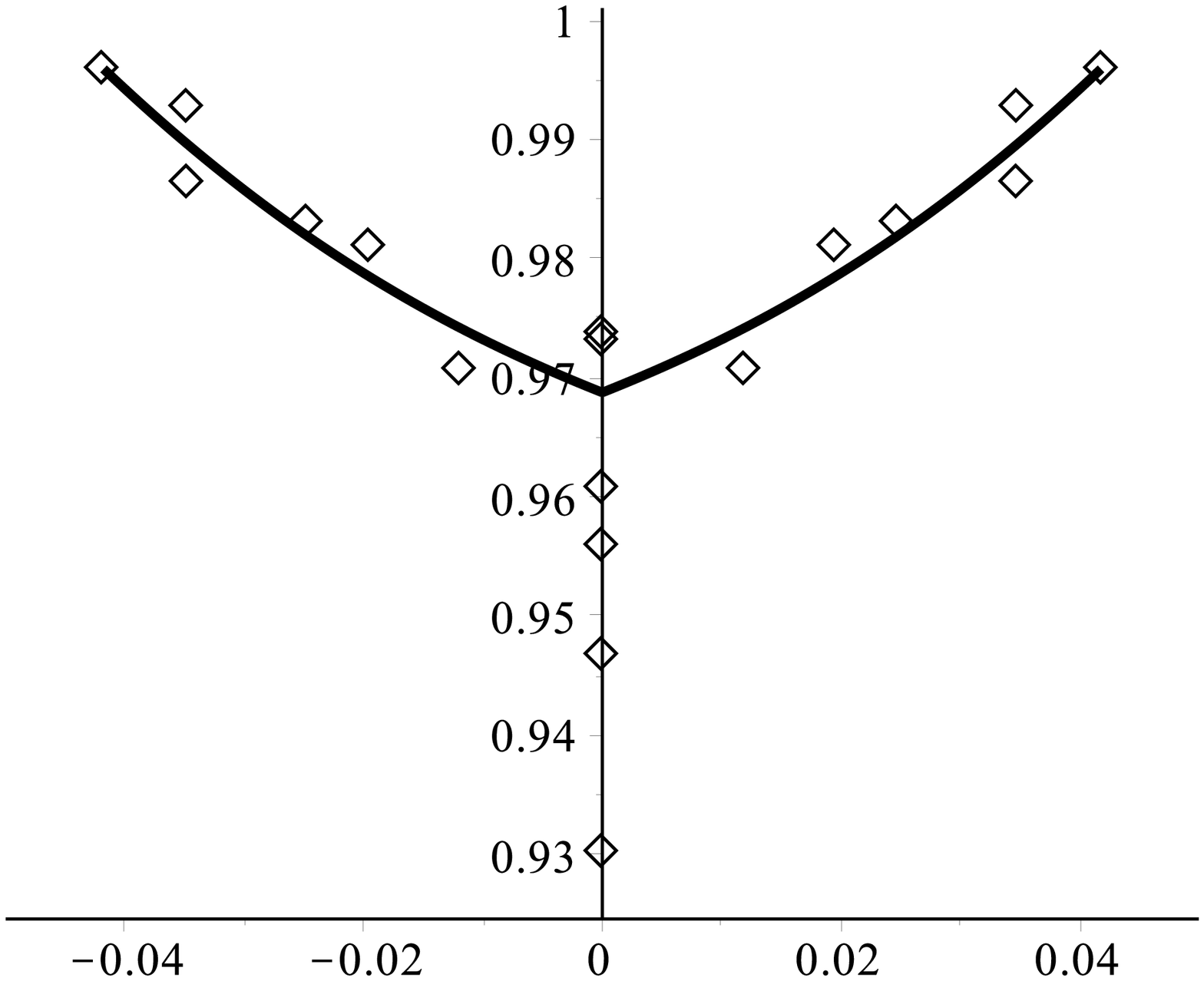} 
\qquad
\includegraphics[height=4.5cm,width=4.5cm]{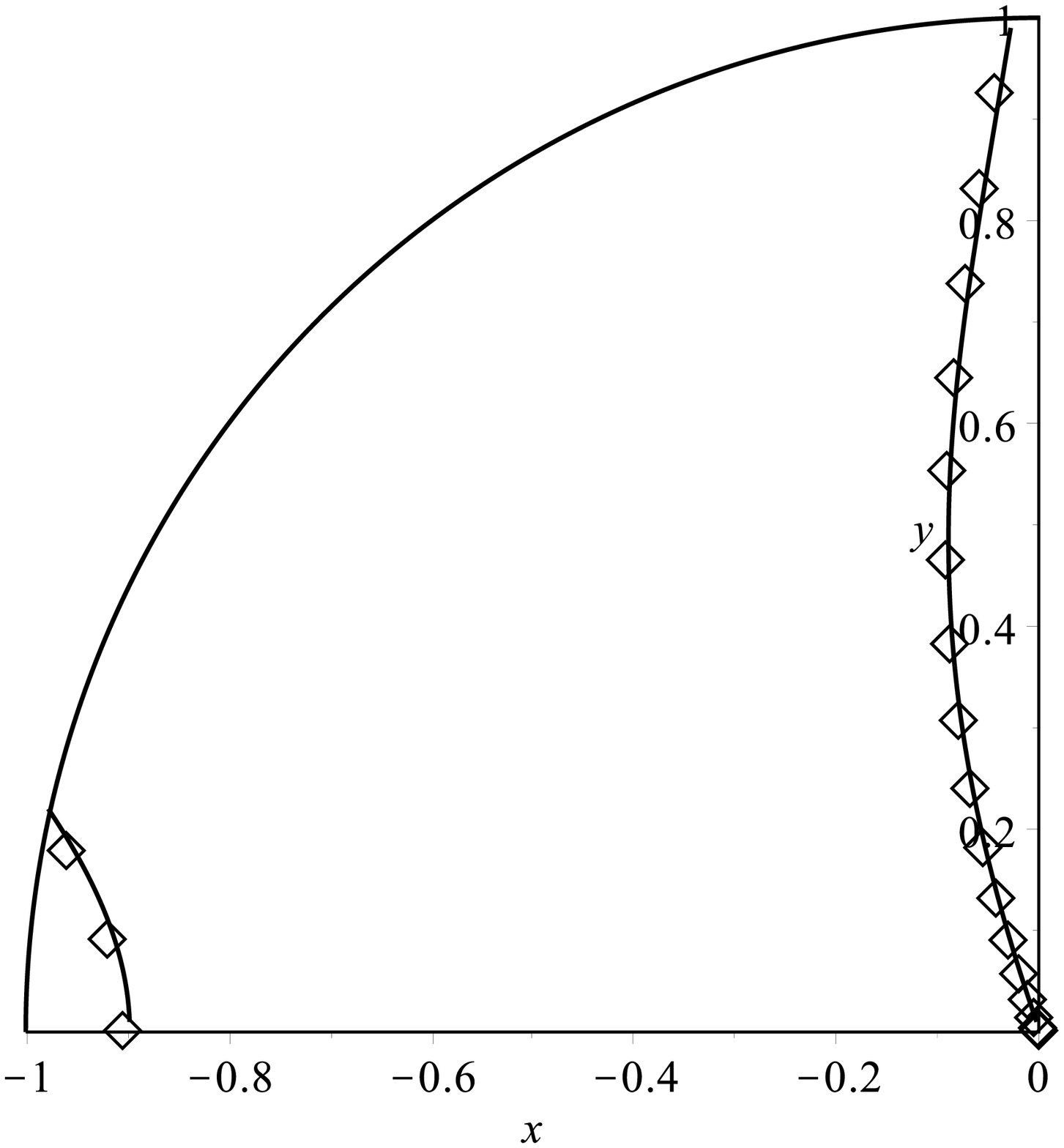} 
\end{center}
\caption{Closeups of  zero attractor inside unit disk:
(a) all parts, zeros of degree 25,000; \,
(b) parts are odd, zeros of degree $2^k$, $k=12,\dots,15$;
\,
(c) parts are congruent to 1 or 2 modulo 3, degree 2500
}
\end{figure}

\section{Further Examples and Conclusion}
This framework for the zeros of partition polynomials is quite deep and general.  
For mathematical convenience we restricted ourselves to the cases that can be fully worked out analytically,
 but this framework generalizes beyond just these cases. 

Here are a few more examples, where we have done a numerical setup and computed potential phase functions  
but have yet to identify 
$\partial R(m).$
In these cases, we  require that the partition parts   are  solutions to $x^2 \equiv 1 \mod p$; that is, they lie in the residue classes $1$ or $p-1$ modulo $p$.
The calculation of the corresponding phase functions   follow  from the same method as those for linear congruences.

Let $(h,k)=1$ with $1 \leq h < k$. Then set
\[
L_{h,k}(z)^2 
=
\frac{ (k,p)^2}{ k^2 p} 
\left(
Li_2(  e_{ (k,p)}( h ) z^{ k/(k,p)} ) + Li_2( e_{ (k, p)}( -h) z^{ k / (k, p)} 
\right) .
\]
When we let $p=3,$ 
\[
L_{h,3k}(z) = L_k(z) = \frac{1}{k\sqrt{p}}\sqrt{2Li_2(z^k)}
\]
and for $k$  not divisible by 3,
\[
L_{h,k}(z) = \frac{1}{k\sqrt{p}}\sqrt{\frac{1}{3}Li_2(z^{3k})-Li_2(z^k)}.
\]
The zero attractor can be seen in Figure (\ref{fig2}) where we have identified the curves which appear to be portions of the
level sets $f_1=f_3$ and $f_2=f_3.$  

For $p=5,$ we obtain the same kind of phenomenon except things cannot be reduced as easily. 
As before, the phase functions are
\[
L_{h,5k}(z) = L_k(z) = \frac{1}{k\sqrt{p}}\sqrt{2Li_2(z^k)}
\]
and for $k$ not divisible by 5,
\[
L_{h,k}(z) = \frac{1}{k\sqrt{p}}\sqrt{\frac{1}{3}Li_2(e_5(1)z^{k})+Li_2(e_5(-1)z^k)}.
\]
Figure (\ref{fig1})  shows with a little bit stronger evidence that the unit disk is made up of three phases $R(1),$  $R(2),$ and $R(3).$

This framework reaches beyond the narrow scope we discussed here.   After all, we first worked out
the asymptotics of polynomials
 for plane partitions 
 whose coefficients are indexed by the trace.
 The corresponding phase functions involved
  a cube root rather than a square root \cite{MR3216005}.  Other more interesting applications would be the polynomial analogue of Wright's partitions into powers 
($a_m=\chi_{m^k}(m)$ for $k\in \mathbb{N}$ fixed) and the polynomial analogue of Cayley double partitions \cite{kaneiwa}
($a_m=p(n)$ where $p(n)$ counts the number of partitions of $n$).

The topological behavior of all the examples studied is also quite intriguing. It seems that every zero attractor drawn so far is a connected.  
Does this generalize the behavior of the eigenvalues of the truncations of  infinite banded Toeplitz matrices \cite{ullman}?  Is it possible to classify zero attractors by their fundamental group properties or even perhaps identify each phase as the interior of a specific set of level curves using the  Jordan curve theorem?  These are all interesting questions which beg an answer for them.

\appendix
 \section{Analysis of the Root Dilogarithm} \label{section:phases}

 For convenience,  we repeat the definition of the root dilogarithm 
 \begin{equation} \label{eq:f_k}
f_{k}(z)=
\frac{1}{k}\Re\left[ \sqrt{{\mathrm{Li}}_{2}(z^{k})}\right] , \quad k \in {\mathbb Z}^+,
\end{equation}
where the real part of the square root is chosen as nonnegative.
By construction each $f_k(z)$ is harmonic on the sectors in the unit disk determined by the rays
$\arg z = \pi j /k$, $0 \leq j < k$, and symmetric about the real axis. Further $f_1(z)=0$ if and only if $z \in [-1,0]$.

From this we learn several facts about $f_k(z)$ which derive from $Li_2(z).$  
For handling the real component of the square root we note that for any $z \in {\mathbb C}$,
 $z \neq 0$, 
\begin{equation}\label{eq:real_part_sq}
\sqrt{z} = \pm \frac{1}{\sqrt{2}}
 \left[ \sqrt{ \Re(z) + \vert z \vert } + i \,  \textrm{sign}( \Im z) 
 \sqrt{ - \Re(z) + \vert z \vert} \right].
\end{equation}
There are two special values of the polylogarithm worthy of note 
$Li_{2}(1) = \zeta(2) = {\pi^2}/{6}$ and 
$Li_2(-1) = -\eta(2) = - {\pi^2}/{12}.$ 
The goal of this appendix is to prove the facts in Section \ref{section:fk}.  We begin with computing the behavior of $f_1(z)$ on circles.

\begin{lemma}\label{lemma:argdecrease}
For a fixed value of $r\in (0,1],$ the function 
$t \mapsto \arg Li_2(re^{it})$ is increasing on $[0,\pi ]$.
\end{lemma}
\begin{proof}
Let $g(z)$ be a univalent function 
on the unit disk 
 normalized by $g(0)=0$ and$g'(0)=1$. Then 
 $g(z)$ is star-like  if
 its derivative satisfies
\[
\frac{ \partial \arg g(r e^{i t})}{\partial t} >0,
\]
for $0<r < 1$ \cite[p. 42]{MR708494}. By a result of  Lewis \cite{MR697137},  
any polylogarithm $Li_s(z)$, $s>0$, is univalent and star-like so the result follows
\end{proof}

By equation (\ref{eq:real_part_sq}),
$f_1(z)=0$ if and only if $\Im Li_2(z)=0$ and $\Re Li_2(z) \leq 0$. Since $Li_2(z)$ is negative on $[-1,0]$,
we find $f_1(z)=0$ if and only if $z \in [-1,0]$.

\begin{lemma}\label{lemma:moddecrease}
For a fixed value of $r\in (0,1],$ the function 
$t \mapsto  \vert  Li_2(re^{it})  \vert$ is decreasing on $[0,\pi ]$.
\end{lemma}
\begin{proof}
Let $h(z) = \sum_{n=1}^\infty c_n z^n$  be a power series  convergent in the unit disk
with $c_n >0$.
For fixed $0<r \leq 1$, the theorem of Fej\'er described in  \cite[p. 513]{MR1626181}
states that $ t \mapsto  \vert h( r e^{it})\vert$ is decreasing provided $\Delta^4( r^n c_n) \geq 0$ for all $n$.
For $Li_2(z)$,  Fej\'er's conditon becomes $\Delta^4( r^n /n^2) \geq 0$ for all $n$.
These inequalities hold because $a_n=r^n/(n+1)^2$ is a moment sequence for the probability
measure on $[0,1]$ with density function $- (1/r) \ln( x/r)$; in fact,
\[
\frac{r^n} { (n+1)^2} = -\int_0^1 \frac{x^n}{r} \log (x/r) \, dx
\]
It follows that $(-1)^j \Delta^j ( a_n ) > 0$; in particular, the required fourth-order difference is also nonnegative.
\end{proof}

\noindent 
These two lemmas combine to prove our first objective.

\begin{proposition}
[Proposition \ref{prop:decreasing} in main text]
For a fixed value of $r\in (0,1],$ the function 
$t \mapsto f_{1}(re^{it})$ is decreasing on $[0,\pi ]$.
\end{proposition}
\begin{proof}
We begin by writing $f_1(r e^{i  t})$ in polar form
\begin{align*}
f_1( r e^{ i t}) 
&= \vert Li_2( r e^{i t}) \vert^{1/2}  \, \cos ( \tfrac{1}{2} \arg Li_2( r e^{i t}) )  .
\end{align*}
Lemmas \ref{lemma:argdecrease} and
\ref{lemma:moddecrease}   show that both factors in this factorization are decreasing functions of $t$.
In particular,  since $ \arg Li_2( r e^{i t}) $ 
is an increasing function of $t$ which is $0$ on the positive real axis and $\pi$ on the negative axis, the function
 $ \cos ( \tfrac{1}{2} \arg Li_2( r e^{i t}) ) $ is decreasing on $[0,\pi]$.
\end{proof}

\noindent 
An immediate application of  this proposition is to obtain  the elementary  bounds we use throughout the appendix.

\begin{lemma}\label{lemma:fkrangebounds}  
For $\vert z \vert \leq 1$, we have the bounds
\\
\begin{enumerate}
\item $\displaystyle \frac{\pi^2}{12} \vert z \vert \leq \vert Li_2(z)\vert \leq \frac{\pi^2}{6} \vert z \vert $, \quad 
\item   $\displaystyle 0\le f_k(z) \leq \frac{ \pi}{ k \sqrt{6}} \vert z \vert^{k/2}$. \label{eq:bound1}
\end{enumerate}
\end{lemma}
\begin{proof}
(1)
Apply the maximum and minimum modulus principle to $Li_2( z)/z$. Since $Li_2(z)$ is univalent
in the unit disk,
it has a unique zero  at $z=0$
which is simple. The maximum occurs at $z=1$ and minimum at $z=-1$ which is observed using Proposition \ref{lemma:moddecrease}.
\\
(2) This follows immediately from part (1) and the fact that $Li_2(1) = \pi^2/6$.
\end{proof}

\subsection{Behavior of the root dilogarithm on the imaginary axis}

As a heurstic, studying $f_k(z)$ on radial lines and circles is how one proves facts about phases.  Proposition \ref{prop:decreasing} is conclusive on $f_k(z)$'s behavior on the circles $\vert z \vert=r  \leq 1$
 but the behavior on radial lines is much harder.  
We can scrape by with knowing how $f_1(z)$ behaves on the ``main" radial lines; the real and imaginary axes.  
On the negative axis $f_1(z)=0,$ and on the positive axis, $f_1(z)= \sqrt{Li_2(z)}$ which is routine.  
So we focus on the imaginary axis and prove

\begin{proposition}\,
[Proposition \ref{f1calculus} in the main text]
The function $r\to f_1(ir)$ is a positive, increasing, concave function of $r\in (0,1).$
\end{proposition}

We  break up this theorem into several lemmas; for convenience, we let
\[
 r\mapsto \theta= \arg( Li_2(ir )) .
\]

We will begin with an improvement  to the bound of Lemma \ref{lemma:fkrangebounds}
 when $z$  lies on the imaginary axis.   
Recall that the real and imaginary parts of $Li_2( i r)$ have the explicit forms 
\cite[Chapter 5]{lewin}:
\begin{align}\label{eq:imaginary_axis}
\Re Li_2( ir) = Li_2(-r^2)/4 \leq 0,
\quad
\Im Li_2( i r ) =
\int_0^r \tan^{-1}(y) \, \frac{dy}{y}  \geq 0  .
\end{align}

\begin{lemma}\label{lemma:f1(ir)_lower_bound} 
 For $0<r \leq 1$,
$ \displaystyle \frac{ \pi \sqrt{2}}{8} \sqrt{r}  < f_1(ir) $.
\end{lemma}
\begin{proof}
 To find this bound for $f_1(ir)$ on $(0,1]$, we start  by explicitly writing
 $f_1(ir)$ using equation (\ref{eq:real_part_sq})
 \begin{align*}
 f_1(ir)
 &=
 \Re \sqrt{ Li_2( ir)}
 =
 \frac{1}{\sqrt{2}} \sqrt{
 \Re Li_2( ir) + \vert Li_2( ir)\vert }
 \\
& >
\frac{1}{\sqrt{2}} \sqrt{
 - \frac{1}{48} \pi^2 r +  \frac{1}{12} \pi^2 r } = \frac{ \sqrt{2}}{8}  \pi \, \sqrt{r}   ,
 \end{align*}
 where 
 \[
 \Re Li_2( ir) = \frac{1}{4} Li_2( - r^2) > \frac{1}{4} Li_2(-1) r = - \frac{1}{48} \pi^2 r, \quad 0<r  < 1  .
\] 
Here we used 
 that $Li_2(-r^2)$ is concave so its secant line on $[0,1]$ gives a lower bound.
\end{proof}

\begin{lemma}
The function $ r\mapsto \theta= \arg( Li_2(ir ))$ is increasing on $[0,1]$.
\end{lemma}
\begin{proof}
For convenience, write $x(r) = \Re Li_2( ir ),$  $y(r) = \Im Li_2(ir) $, and
$dy/dx = \tan \theta$. We will use the dot notation for derivative. By equation
(\ref{eq:imaginary_axis}), we note that
  $\dot{x}$, $\ddot{x}$, and $\ddot{y}$ are negative  while $\dot{y}$ is positive.
Consequently,  the sign of the desired derivative is positive:
\[
\frac{d}{dr}
\tan \theta
=
\frac{ \dot{x} \, \ddot{y} - \ddot{x} \, \dot{y}}{ ( \dot{x} )^2} > 0
\]
\end{proof}

\begin{remark}\label{remark:f1(ir)}
{\rm 
Since $\theta(r)$ is increasing, $\Re Li_2(ir) < 0$, and $\Im Li_2(ir)>0$,
the range of $\theta(r)$ must lie in $[\pi/2,3 \pi/2]$.
By equation (\ref{eq:imaginary_axis}),  
$Li_2(i) $ is given explicitly by $ - \pi^2/48 + i G$, where
$G$ is the Catalan constant $\sum_{n=0}^\infty (-1)^n/(2n+1)^2$.
We find that $\arg Li_2(i) = \pi - \arctan( 48 G / \pi^2)  \simeq 1.79161$
so the range of $\theta(r)$ is exactly $[\pi/2 , \arg Li_2( i )]$ which lies inside
$[ \pi/2, \pi/3]$.
Further,  using equation  (\ref{eq:real_part_sq}), 
we record  the explicit form for $f_1(i)$ 
\begin{align*}
f_1(i ) 
=
\frac{1}{4 \sqrt{6}} \sqrt{ - \pi^2+  \sqrt{ \pi^4+ 2304 G^2}} 
\end{align*}
so  we can verify $f_1( i ) < f_2(1) = \pi  \sqrt{6} /12$ by direct calculation.

To complete the proofs in the section, we need  the bounds
\begin{align*} 
\frac{1}{2}
  \leq \cos\left(\frac{\theta(r)}{2}\right) \leq  \frac{\sqrt{2}}{2} 
\leq
   \sin\left(\frac{\theta(r)}{2}\right) \leq \frac{\sqrt{3}}{2} . 
   \end{align*}
If needed, these bounds can be tighten by using 
 \begin{align*}
  \cos(\theta(1)/2) 
  &=
   \sqrt{ (a- \pi^2)/(2a)} \simeq 0.62 488,
  \\
 \sin\left( \theta(1)/2\right) 
 &=  
  \sqrt{ (a+ \pi^2)/(2a)}   \simeq 0.78071
 \end{align*}
with $a= \sqrt{ \pi^4+2304 G^2}.$
} \end{remark}

\begin{lemma}
$f_1(ir)$ is increasing on $[0,1]$.
\end{lemma}
\begin{proof}
We write out the real part of the derivative
\[
\frac{d}{dr} \sqrt{ Li_2( i r) }
=
-  \,
\frac{  \ln (1- ir) }{  2  r \sqrt{ Li_2( ir) } } .
\]
using the polar form of
 $\displaystyle
 \sqrt{Li_2( ir)  } = \vert Li_2( ir)\vert^{1/2} e^{ i \theta(r)/2} 
 $
 to get
\[
\frac{d}{dr} \Re \sqrt{ Li_2( i r) }
=
-  \, \Re\left(
\frac{   \frac{1}{2} \ln(1+r^2) - i \arctan(r)   }{  2  r \vert Li_2( ir)\vert^{1/2}  e^{i \theta(r)/2}  }
\right)  .
\]
Now its sign is the same as the sign of
\[
\sin(\theta(r)/2) \arctan(r) - \frac{1}{2}  \cos(\theta(r)/2) \ln(1+r^2)
\]
which is bounded below by
\[
\frac{\sqrt{2}}{2} 
\left( 
\arctan(r) - \frac{1}{2}   \ln(1+r^2)
\right)
=
\frac{\sqrt{2}}{2}  \,
\int_0^r \frac{1-t}{ 1+t^2} \, dt \geq 0 
\]
since $\theta(r)/2 $ lies in the interval $ [\pi/4, \pi/3]$.
\end{proof}

\begin{lemma}\label{lemma:f_1(ir)_concave}
For $r\in (0,1),$ $f_1(ir)$ is a concave function.
\end{lemma}
\begin{proof}
We use  polar form again this time  for the second derivative of $f_1(ir)$. 
For the sake of notation, we introduce 
$g_1(r)$ and $g_2(r)$ by
\[
\frac{d^2}{dr^2}
\Re \sqrt{ Li_2( ir)}
=
\Re\left( g_1(r) + g_2(r) \right)
\]
where
\begin{align*}
g_1(r)
&=
- \frac{1}{4}
\frac{ ( \frac{1}{2} \ln(1+r^2) - i \arctan(r) )^2 }{ Li_2(ir)^{3/2} r^2 } ,\\
g_2(r)
&=
\frac{ i/2}{ Li_2( ir)^{1/2} (1- ir) r} 
+
\frac{1}{2}
\frac{ \frac{1}{2} \ln(1+r^2) - i \arctan(r) }{ Li_2(ir)^{1/2} r^2} .
\end{align*}
  Since $\pi/2 \leq \theta(r)\leq  2 \pi/3$,  we have
  $\sin( 3 \theta(r)/2) $ is positive while $\cos( 3 \theta(r)/2) $  is negative. 
  It is enough  show that the signs of $\Re g_1(r)$ and $\Re g_2(r)$ are negative.
Now the sign of $\Re g_1(r)$ is the same as the sign of
\begin{align*}
-\Re
&
\left[
\big(   \tfrac{1}{4} (\ln(1+r^2))^2 - i (\arctan(r))^2  \big) \,  \left( \cos ( 3\theta(r)/2) - i \sin( 3\theta(r)/2) \right)
\right]
\\
&
\qquad =
\cos( 3 \theta(r)/2) 
\left\{
 ( \arctan(r))^2   - \tfrac{1}{4} ( \ln(1+r^2))^2 
\right\}
\\
& \qquad \qquad
-
\sin(3 \theta(r)/2) \ln(1 + r^2) \arctan( r) 
\end{align*}
which is negative since $\sin(3 \theta(r)/2) \ln(1 + r^2) \arctan( r) $ is positive and
\begin{align*}
 ( \arctan(r))^2   - \tfrac{1}{4} ( \ln(1+r^2))^2 
 &=
 \int_0^r \int_0^r \frac{ ( 1-s)(1+t)}{ (1+s^2) (1+t^2)} \, ds dt >0
\end{align*}
as well. The sign of $\Re g_2(r)$ is the same as the sign of 
\begin{align*}
&
\frac{1}{ 1+r^2} \left( - r \cos( \theta(r)/2) + \sin( \theta(r)/2) \right)
\\
& \qquad
+
\frac{1}{r} 
\left(
\tfrac{1}{2}  \ln(1+r^2) \cos( \theta(r)/2) - \arctan(r) \sin( \theta(r)/2)
\right)
\\
&\,\, =
\frac{1}{r} \int_0^r \frac{  r \cos( \theta(r)/2) + \sin( \theta(r)/2) }{ 1+r^2} \, dt
+
\frac{1}{r} \int_0^r \frac{t \cos(\theta(r)/2) - \sin( \theta(r)/2)}{ 1+t^2} \, dt
\\
&\qquad\qquad
=
\frac{ \cos( \theta(r)/2)}{r}
\int_0^r \left( \frac{t}{1+t^2} - \frac{r}{1+r^2} \right) \, dt
\\
&\qquad \qquad \qquad
+
\,\,
\frac{ \sin( \theta(r)/2)}{r} \int_0^r \left( \frac{1}{ 1+r^2} - \frac{1}{ 1+t^2} \right) \, dt 
\end{align*}
which is negative
since both integrals are negative while $ \sin( \theta(r)/2 )$ and $ \cos( \theta(r)/2 )$ are positive.
\end{proof}

Although the second derivative of $f_1(r)$ changes sign, $f_2(r)$ is convex.

\begin{lemma}\label{lemma:f_2_convex}
$f_2(r)$ is convex.
\end{lemma}
\begin{proof}
The second derivative of $f_2(r)$ is
\[
\frac{ -(  \ln( 1-r^2) )^2 }{ r^2 ( Li_2(r^2) )^{3/2}}
+
\frac{ \ln(1-r^2)}{ r^2 \sqrt{ Li_2(r^2)}}
+ 
\frac{2}{ (1-r^2) \sqrt{ Li_2(r^2)}}
\]
which has the same sign   as
\[
h(r) = 2r^2 Li_2(r^2) - (1-r^2) ( \ln(1-r^2)  )^2    + (1-r^2) \ln(1-r^2)    .
\]
It is elementary that the Taylor expansion of each term above has only even powers with nonnegative  coefficents starting with $r^6$.
\end{proof}

We record  that $f_k(r)$ is convex for $k \geq 2$.

\begin{proposition} 
[Proposition \ref{positivitybeta} in the main text]
There is a unique solution $\beta$ to $f_1(i r)= f_2(r)$ in $(0,1]$
further, it satisfies  $3/4< \beta <1$.  For $r \in (0, \beta),$  
$f_1(i r)> f_2(r)$ while for  $r\in (\beta,1),$   $f_1(i r)< f_2(r) .$
\end{proposition}
\begin{proof}
Define $h(r)=f_2(r)-f_1(ir)$ with $\in (0,1).$   By the last two Lemmas \ref{lemma:f_1(ir)_concave}
and
\ref{lemma:f_2_convex}, 
$h(r)$ must be a convex function and therefore has at most one critical point $h'(r)=0$ in $(0,1)$.
This also implies that it can have at most two roots since if $h(x_1)=h(x_2)=h(x_3)=0$ for some $x_1<x_2<x_3$ then by Rolle's theorem $h'(\xi_1)=h'(\xi_2)=0$ for $\xi_1\neq \xi_2.$  We then identify that $h(0)=0$ trivially and from the     inequalities in Lemma \ref{lemma:fkrangebounds}, we find that
\[
f_2(r) < \frac{ \pi r }{ 2 \sqrt{6}} < \frac{ \pi \sqrt{ 2r}}{ 8} < f_1(ir)
\] 
for  $0<r < 3/4$. Further, $f_2(i) > f_1(i)$
by Remark \ref{remark:f1(ir)}.  
Hence, by continuity
there exists $\beta \in  (3/4,1)$ such that $f_2( i \beta) = f_1( i \beta)$.

\end{proof}

\subsection{Generic Dominance of $f_k(z)$ on the Unit Disk}

The last major goal of this appendix is to show $f_1(z)$, $f_2(z)$, and $f_3(z)$ dominate  all the other $f_k(z)$  in the unit disk and determine where this dominance holds.  In particular we need to show
\begin{theorem}\,
[Theorem \ref{fkdomanancetheorem} in the main text]
For $0 < \vert z \vert \leq 1$, 
\begin{enumerate}
\item 
$ f_k(z)\le f_k(\vert z \vert)\le f_2(\vert z \vert) < f_1(z),$  $k\geq 2,$  $\vert\arg z\vert \le \pi/3,$
\item 
$ f_k(z)\le f_k(\vert z \vert) < f_1(z),$  $k\geq 3,$  $\vert\arg z\vert \le \pi/2$,
\item 
$ f_k(z) < f_1(z)$, $k \geq 2$, $0 \leq  \arg z \leq \pi/2$,
\item
$ f_k(z) < \max[  f_1(z), f_2(z), f_3(z)],$   $  k \geq 4$, $\pi/2 \leq \arg z \leq \pi$.
\end{enumerate}
\end{theorem}
As  with the previous theorems, we will prove this using a series of lemmas and propositions.  Part (1) of this theorem is an immediate consequence of the next lemma. 

\begin{lemma}\label{lemma:bound2} 
The following are true:
\begin{enumerate}
\item Let $0<\vert z \vert<1$. If $\Re Li_2(z) >0$, then  $f_k(z)<f_2( \vert z \vert) < f_1(z)$ for $k \geq 2$.
\item For every $r\in (0,1)$ and $\theta \in (0,\pi/3),$ $\Re Li_2( r e^{i\theta}) >0.$ 
\end{enumerate}
\end{lemma}
\begin{proof}
We begin with the simple bound
\begin{align*}
f_1(z)
 &=
  \frac{1}{ \sqrt{2}} \sqrt{ \Re Li_2(z) + \vert Li_2(z) \vert } 
  >
   \frac{1}{\sqrt{2}} \sqrt{ \vert Li_2(z)\vert} 
   \geq
   \frac{1}{ \sqrt{2}} \sqrt{ -Li_2(-1) \vert z \vert }
   \\
   &=
   \frac{1}{ \sqrt{2}} \sqrt{ \frac{1}{2} Li_2(1)  \vert z \vert }
   =
   \frac{1}{2} \sqrt{ Li_2(1) \vert z \vert }
   \geq
   \frac{1}{2} \sqrt{ Li_2( \vert z \vert) }  \geq f_2(z) .
\end{align*}
To check that $\Re Li_2( r e^{i \pi/3}) >0$,
we use the identity that $\Re Li_2( r e^{i \pi/3}) = \frac{1}{6} Li_2( - r^3) - \frac{1}{2} Li_2( -r)$ \cite[p. 133]{lewin} and Lemma \ref{lemma:argdecrease}.
\end{proof}

By symmetry we need only consider the upper unit disk.  For convenience we divide the upper disk into three sectors:
$\arg(z) \in [0,\pi/2]$,  $\arg(z) \in [\pi/2,3\pi/4]$, and $\arg(z) \in [3\pi/4,\pi]$.

\begin{proposition}\label{prop:general_bounds}
For $0 < r \leq 1$, we  have the following bounds:
\\
(a)  $ f_3(r)   < f_1( ir)$,
\,
(b)
  $ f_5(r) < f_2(r e^{ i 3 \pi/4})$,
 \,
 (c)
  $ f_7(r) < f_3( ir)$.
\end{proposition}
\begin{proof}
Note that $f_2(r e^{ i 3 \pi/4}) = f_1(  i r^2  )/2$ and $f_3( ir) = f_1(ir)/3$.
In other words, the desired inequalities reduce to the following easily to verify bounds which follow from Lemma \ref{lemma:f1(ir)_lower_bound} and Equation  (\ref{eq:bound1}):
\begin{align*}
f_3(r) \leq \frac{\pi}{3 \sqrt{6} } r^{3/2} & < \frac{\pi  \sqrt{2}} { 8}r^{1/2} \leq f_1(ir) ,
\\
f_5(r) \leq \frac{\pi}{5\sqrt{6} } r^{5/2} & < \frac{ \pi  \sqrt{2}}{  16} r \leq f_2( re^{i 3 \pi/4}) ,
\\
f_7(r) \leq \frac{ \pi}{ 7 \sqrt{6} } r^{7/2} &< \frac{ \pi   \sqrt{2}}{  24} r^{3/2} \leq f_3( i r) .
\end{align*}
\end{proof}

By the above Proposition  
together with Proposition \ref{prop:decreasing} and  the bounds $f_k(z) \leq f_k( \vert z \vert) \leq f_j(\vert z \vert)$,
 for $j \leq k$,
we can extend the bounds on radial lines to the three  above designated sectors inside the unit disk.  

\begin{corollary}
Let $0<\vert z \vert \leq 1$.
\\
(1)
For $k \geq 3$,
$f_k (z) \leq  f_k( \vert z \vert) < f_1(z)$, $0 \leq \arg(z) \leq \pi/2$.
\\
(2)
For $k \geq 5$, $f_k(z) \leq f_k( \vert z \vert) <  f_2(z)$, $3 \pi /4 \leq \arg(z) \leq \pi$.
\\
(3)
For $k \geq 7$, $f_k(z) \leq f_k( \vert z \vert) <  f_3(z)$, $\pi/2 \leq \arg(z) \leq 3 \pi/4$.
\end{corollary}
\noindent Part (2) of Theorem \ref{fkdomanancetheorem} is part (1) of this corollary.

To sum up, we know that, for $0<\vert z \vert\leq 1$, $f_k(z) < f_1(z)$
with $\arg(z) \in [0, \pi/2]$, $k\geq 3$ and $f_k(z) < \max( f_2(z),f_3(z))$ with $\arg(z) \in [\pi/2,\pi]$, $k\geq 7$.
In order to complete the proof of parts (3) and (4) of Theorem  \ref{fkdomanancetheorem}, we
  must  establish  the five remaining
cases which  requires working throught the remaining   three special sectors.

\subsection{Refining the Dominance of $f_k(z)$ on the Open Unit Disk}

We now show  for $k\geq 4,$ $f_k(z) < \max \left(f_1(z),f_2(z),f_3(z)\right)$ for $0<\vert z \vert \leq 1$.
The maximum modulus principle applies in  these exceptional cases since the functions $f_k$ are harmonic
on the appropriate sector $S$ in the unit disk.
 In particular, we need only to check dominance on the two boundary  radial lines,
say $\arg(z) = \alpha$ or $\beta$ and on the boundary arc of the unit circle $e^{ i t }$, 
$\alpha \leq t \leq \beta$.
On the unit circle, $Li_2(e^{it})$ has an explicit representation:
\[
 Li_2( e^{it}) = r(t) - i Cl_2(t)
 \]
 where
\[
r(t) = \frac{\pi^2}{6}  - \frac{ t ( 2\pi -t)}{4}, 
\quad
 Cl_2(t) = \int_0^t \ln [ 2 \sin ( s/2)] \, ds 
\]
(see \cite[p. 101]{lewin}). $Cl_2(t)$ is called the Clausen integral
and has special values $Cl_2(\pi/2)=-G$, the Catalan constant, and $Cl_2(\pi)=0$.
  By construction, $\Im Li_2( e^{ it}) = -  Cl_2(t)$ is a nonnegative concave decreasing function
on $[\pi/2,\pi]$ so  its secant line 
\[
s(t)=2G ( \pi -t)/\pi
\]
 gives a  lower bound on $[\pi/2,\pi]$.  The tangent line $L(t)$ at $s=\pi/2$ yields an upper bound
 where $L(t) = s/2+ (1/2) \ln(2) - \pi/4$.
Given the explicit form of the real part of the square root (equation (\ref{eq:real_part_sq})),
we have a lower bound for $f_1( e^{it})$ on $[\pi/2 , \pi]$:
\[
f_1( e^{it}) \geq \frac{1}{ \sqrt{2}} 
\sqrt{ r(t) + \sqrt{ r^2(t) + s^2(t)}} .
\]

We obtain the following inequalites by direct calculation using the secant  lower bound and the
exact forms of $f_1(1)$ and $f_1(i)$:
\begin{equation}\label{eq:direct_inequalities}
f_1(i)/4 < f_1( e^{ i 3 \pi/4}), 
\quad
f_1(1)/4 < f_1( e^{ i 2 \pi/3}),
\quad 
f_1( e^{ i 2 \pi/3})/5 < f_1( e^{ i 3 \pi/4})
\end{equation}
as well as 
\begin{equation}\label{eq:prop15}
f_1( e^{ i 2 \pi/3}) /5 < f_1( e^{ i 4 \pi/5}) 
\end{equation}
needed in the proofs of Propositions \ref{prop:13} and \ref{prop:15} below.
Further, using the lower bound $\ln[2\sin(\pi/4)]$ for $\ln[2\sin(s/2)]$ on $[\pi/2,2\pi/3]$,
we obtain
\begin{align*}
-Cl_2( 2 \pi/3) 
&=
 - \int_0^{\pi/2} \ln [ 2\sin(s/2) ] \, ds -  \int_{\pi/2}^{2\pi/3} \ln [ 2\sin(s/2) ] \, ds
\\
&= 
G - \int_{\pi/2}^{2\pi/3} \ln [ 2\sin(s/2) ] \, ds < G - \ln[ 2\sin( \pi/4)](2\pi/3-\pi/2)
\\
&=
G -  \frac{\pi \ln 2}{12} .
\end{align*}
Substituting, we find that
equation (\ref{eq:prop15}) now follows.

Note that  that $f_1( e^{ i 2 \pi/3}) < f_3( e^{ i 2 \pi/3}) = f_1(1)/3$ by the tangent line upper bound
$L(t)$
for $-Cl_2( 2 \pi/3)$ on $[\pi/2,2\pi/3]$.  Since $f_2( e^{ i 2 \pi/3}) = f_1( e^{ i 2 \pi/3})/2$, we find that
\[
f_2( e^{ i 2 \pi/3})<f_1( e^{ i 2 \pi/3})< f_3( e^{ i 2 \pi/3}),
\]
so all three functions $f_1(z), f_2(z)$, and $f_3(z)$ occur in describing phases.

We  will also make use of  the symmetry relation  $f_1( e^{ i ( \pi +t ) }) = f_1( e^{ i (\pi - t)})$ on $[0,\pi]$.

\begin{lemma}
For $0<r\leq 1$, $ f_3( r e^{ i 5 \pi/6}) < f_2( r^{ i 5 \pi/6})$.
\end{lemma}
\begin{proof}
We record that  $f_3( r e^{ i 5 \pi/6}) = f_1(  i r^3)/3$ and $f_2( r e^{ i 3 \pi/4}) = f_1( i r^2)/2$.
Since $f_2( r e^{ it})$ is increasing for $t \in [\pi/2,\pi]$ and $f_1(ir)$ is also increasing, the result follows.
\end{proof}

\subsubsection{Dominance of $f_k(z)$ on the Quarter Disk $\arg(z) \in [0, \pi/2]$}
On the quarter disk, we already know $f_k(z)\leq \max \left(f_1(z),f_2(z)\right)$
with $k \geq 3$.  We must reduce this to show $f_2(z) < f_1(z)$.  This also concludes part (3) in Theorem \ref{fkdomanancetheorem}.

\begin{proposition}
\label{f1-f2-first-quadrant}
For $0< \vert z \vert \leq 1$ and $0 \leq \arg(z) \leq \pi/2$, 
\[
f_2(z) < f_1(z) .
\]
\end{proposition}
\begin{proof} 
Since both $f_1$ and $f_2$ are harmonic 
on the open quarter unit
disk,
we can use the maximum modulus principle. 
On the radial boundary lines, $f_2(r) < f_1(r)$ and $0 = f_2(ir) < f_1(ir)$, $r \in (0,1]$.
On the arc $e^{it}$, $t \in [0,\pi/2]$, 
  both $f_1( e^{it} )$ and $f_2( e^{it} )$ are decreasing functions by Proposition \ref{prop:decreasing}. 
On $t\in [0,\pi/4],$
\[
f_2(e^{it}) \le f_2(1) \le f_1(e^{i  \pi/4  }) < f_1(e^{it}).
\]
 as $f_2(1) < f_1( e^{i  \pi/4})$ by Lemma \ref{lemma:bound2}.
For $t\in [\pi/4,\pi/2]$:
\[
f_2(e^{it}) \leq f_2(e^{i {\pi}/{4}})=\frac{1}{2}f_1(i) \le f_1(i) < f_1(e^{it}).
\]
\end{proof}

\subsubsection{Dominance of $f_k(z)$ on the sector $\arg(z) \in [3\pi/4,\pi]$} On this sector we have for $k\geq 5,$
  \[
  \max \left(f_1(z),f_2(z),f_3(z), f_4(z)\right)>f_k(z).
  \]
   To refine this result, we show $f_4(z) < f_2(z).$

\begin{proposition}
On the sector $\arg z \in [3 \pi/4,\pi]$ and $0< \vert z \vert\leq 1$, 
$f_4(z) < f_2(z)$.
\end{proposition}
\begin{proof}
The sector $S$ has boundary lines with $\alpha=3 \pi/4$ and $\beta= \pi$.  When $\arg(z)=3\pi/4$, $f_4(z)=0$
while $f_2( r e^{ i 3 \pi/4}) = f_1( i r^2)/2 >0$ when $0<r\leq 1$. For $\beta= \pi$, 
$0<f_4(-r)=f_4(r) < f_2( -r)=f_2(r)$,
$0<r \leq 1$. On the arc $e^{it}$, $3\pi/4 \leq t \leq \pi$,  both $f_2( e^{it})$ and $f_4( e^{it})$ are increasing by Propositon \ref{prop:decreasing}. We check the chain of inequalities
\[
f_4( e^{ i 5 \pi/6}) = f_1( e^{ i 2 \pi/3})/4 <   f_2( e^{ i 3 \pi/4}) =f_1( i )/2
\]
by Propositon \ref{prop:decreasing}
and 
\[
f_4(1) = (1/2) ( f_1(1)/2)  < f_1( e^{ i  \pi/3  })/2 =  f_2( e^{ i 5 \pi/6}) 
\]
by Lemma \ref{lemma:bound2}.
Thus
$$
f_4(e^{it}) < f_4(1) < f_2(e^{i\frac{5\pi}{6}}) < f_2(e^{i\frac{3\pi}{4}}) <f_2(e^{it}).
$$
\end{proof}

\subsubsection{Dominance of $f_k(z)$ on the sector  $\arg(z) \in [\pi/2,3\pi/4]$   }
As in the previous two sections, we have to reduce the number of relavent $f_k(z).$  Things are a bit more fine tuned in this section however. There are three cases.

\begin{proposition}\label{prop:13}
On the sector with $ \arg z  \in [ \pi/2 , 3 \pi/4]$, 
$f_ 4(z) < f_1(z)$, $z \neq 0$.
\end{proposition}
\begin{proof}
The sector $S$ has boundary lines with $\alpha=\pi/2$ and $\beta=3\pi/4$.
We need to recall Theorem \ref{fkdomanancetheorem} part (2) (which is already proven) that $f_3(r)\le f_1(ir).$
 When $\arg(z)=\pi/2$,
$f_4(  i r) = f_4(r) < f_3(r) \le  f_1( i r)$,  $ 0< r \leq 1$. When $\arg(z)=3\pi/4$,
$f_4(r e^{ i 3 \pi/4})=0$ while $f_1(r e^{ i 3 \pi/4})>0$, $0<r \leq 1$.  
Hence $f_4(z) < f_1(z)$ on the radial boundary lines.

On the arc $e^{i t}$, $t \in [ \pi/2, 3\pi/4],$ we can use Proposition \ref{prop:decreasing} to show
both $f_1( e^{it})$ and $f_4( e^{it})$ are decreasing.  Using the secant line
bounds for $f_1(1)/4<f_1( e^{ i 2 \pi /3})$ and $f_1(i)/4<f_1( e^{ i 3 \pi/4})$, we find  that the required inequalities hold on $t\in [\pi/2,2\pi/3]$:
\[
f_4(e^{it})\le f_4(i) = f_1(1)/4 = \frac{ \pi}{ 4 \sqrt{6}}  < f_1( e^{ i 2 \pi/3})  \leq f_1( e^{ i t}) 
\]
and for  $t\in [2\pi/3,3\pi/4]$
\[
f_4(e^{it})\le f_4( e^{ i 2 \pi/3}) = f_1( e^{ i 2 \pi/3})/4 < f_1(i)/4 < f_1( e^{ i 3 \pi/4}) \leq f_1(e^{it}) 
\]
where $f_1( e^{ i 2 \pi/3})/4 < f_1(i)/4$ holds since $f_1(e^{it})$ is decreasing
and $ f_1(i)/4 < f_1( e^{ i 3 \pi/4})$ by  equation  (\ref{eq:direct_inequalities}).

\end{proof}

\begin{proposition}
On the sector with $\arg z \in [ \pi/2, 3\pi/4]$, $f_6(z) < f_3(z)$, $z \neq 0$.
\end{proposition}
\begin{proof}
We need to consider two sectors $S_1$ and $S_2$ so $\arg(z) \in [ \pi/2,2 \pi/3]$ and $[2 \pi/3, 3\pi/4]$ so
$f_6(z)$ will be harmonic in each.
The boundary lines of the sector $S_1$ are  $\alpha=\pi/2$ and $\beta=2 \pi/3$.
For $\arg(z)  =\pi/2$, $f_6( ir) = f_1( - r^6)=0$ while $f_3( ir) = f_1( i r^3)>0$ for $r\in (0,1]$.
For $\arg(z) = 2 \pi/3$, $f_6( r e^{ i 2 \pi/3}) = f_1(r)/6$ while $f_3( r e^{ i 2 \pi/3}) = f_1(r^3)/3$.
For $\arg(z) = 3 \pi/4$, 
\[
f_6( r e^{ i 3 \pi /4})  =  \frac{1}{6} f_1( i r^6) < \frac{1}{3} f_1( ir^3) < \frac{1}{3} f_1( r^3 e^{ i \pi/4}) 
= f_3( r e^{ i 3 \pi/4}) 
\]
since $f_1( ir)$ is increasing in $r$ by Proposition \ref{f1calculus}  and $f_1(z)$ is decreasing in the argument of $z$ by Proposition \ref{prop:decreasing}. Hence $f_6(z)  \leq f_3(z)$
are the radial lines. 

On the arc for $S_1$, both $f_3( e^{ it})$ and $f_6( e^{ it})$ are increasing on $[\pi/2, 2 \pi/3]$ by Proposition \ref{prop:decreasing}. It convenient
to use two subintervals $[ \pi/2, 7 \pi/12]$ and $[ 7 \pi/12, 2 \pi/3]$ in this case. We need to check that
\[
f_6( e^{ i 7 \pi /12}) < f_3( e^{ i \pi/2}), \quad f_6( e^{ i 2 \pi/3}) < f_3( e^{ i  7 \pi /12}) .
\]
Now
\begin{align*}
f_6( e^{ i 7 \pi /12})  &= f_1( e^{ i \pi /2 })/6 = f_1( i) /6 < f_1( i ) /3 =  f_3( e^{ i \pi/2})  ,
\\
f_6( e^{ i 2 \pi/3}) &= f_1( 1)/6 < f_3( e^{ i  7 \pi /12})  = f_1( e^{ i \pi/4}) /3
\end{align*}
since $\Re Li_2( e^{i \pi/4})>0$
and Lemma \ref{lemma:bound2}.

On the arc   $[2\pi/3, 3 \pi/4]$ for $S_2$, both $f_3(e^{it})$ and $f_6( e^{it})$ are decreasing.
The required inequality $f_6( e^{ i 2 \pi/3}) < f_3( e^{ i 3 \pi/4})$ holds  because 
$f_3( e^{ i 3 \pi/4}) =  f_1( e^{ i \pi/4}) /3$; so this inequality  reduces to the last inequality above.
\end{proof}

\begin{proposition}\label{lemma:weak} \label{prop:15}
On the sector $\arg z \in [ \pi/2, 5 \pi/8]$, $f_5(z) < f_1(z)$, $z \neq 0$.
\\
(b)
On the sector $\arg z \in [3 \pi/5, 4 \pi/5]$, $f_5(z) < f_2(z)$, $z \neq 0$.
\end{proposition}
\begin{proof}
(a)
Since $f_4(z) < f_1(z)$ on $[\pi/2,3 \pi/4]$, it is enough to show $f_5(z) < f_4(z)$. 
For $t \in [ \pi/2, 5 \pi/8]$, $f_4( re^{it})$ is decreasing while $f_5( r e^{ it})$ is decreasing on $[\pi/2,3\pi/5]$
and increasing on $[3\pi/5,5\pi/8]$.  For fixed $r \in (0,1]$,  the minimum value of $f_4( re^{it})$ is $f_4(re^{i 5\pi/8})
= f_1( i r^4)/4$ while the maximum value of $f_5(r e^{it})$ is $f_5( i r) = f_1( i r^5)/5$. 
Since $f_1( ir)$ is an increasing function, the desired inequality holds.
\\
(b)
We apply the maximum modulus principle on a sector $S$
with boundary radial lines $\arg(z) = \alpha$ or $\beta$ where
$\alpha=3\pi/5$ and $\beta=3 \pi/4$.

For $\arg(z) =  3 \pi/5$, $f_5(z)=0$ while $f_2( r e^{ i 3 \pi/5})>0$ with $r \in (0,1]$.
For $\arg(z)= 4 \pi/5$,  we see by part (2) of Proposition \ref{prop:general_bounds} that
\[
 f_5(r) <  f_2( r e^{ i 3\pi/4})  < f_2( r e^{ i 4 \pi /5}) .
\]
On the arc $e^{it}$, $3 \pi/5 \leq t \leq 4 \pi/5$, both $f_2(e^{it})$ and $f_5(e^{it})$ are increasing 
so it is enough to check the two evaluations:
\[
f_5( e^{ i 2 \pi/3}) < f_2( e^{ i 3 \pi/5}), \quad f_5( e^{ i 3 \pi/4}) < f_2( e^{ i 3 \pi/4 } ) .
\]
We just saw that the last inequality automatically holds.
For the remaining inequality, consider
\[
f_5( e^{ i 2 \pi/3}) = f_1( e^{ i 2 \pi/3})/5,  \quad f_2( e^{ i 3 \pi/5}) = f_1( e^{ i 4  \pi/5}) /2
\]
then we can  use equation (\ref{eq:prop15}).

\end{proof}

\end{document}